\pdfoutput=1

\documentclass{article}

\usepackage{breakcites}
\usepackage[utf8]{inputenc}
\usepackage[english]{babel}
\usepackage[document]{ragged2e}
\usepackage{blindtext,titlefoot}
\usepackage{amsmath,amssymb,amsfonts,graphicx,nicefrac,mathtools}
\usepackage{mathtools}
\usepackage[body={7in, 9in},left=1in,right=1in]{geometry}
\usepackage{graphicx}
\usepackage{subcaption}

\usepackage{amsthm}
\usepackage{mathrsfs}

\usepackage{hyperref}
\usepackage{color}
\usepackage[dvipsnames]{xcolor}
\usepackage{enumitem}
\usepackage{float}
\setlist{nolistsep}
 
\usepackage{mathtools, stmaryrd}
\usepackage{xparse} \DeclarePairedDelimiterX{\Iintv}[1]{\llbracket}{\rrbracket}{\iintvargs{#1}}
\NewDocumentCommand{\iintvargs}{>{\SplitArgument{1}{,}}m}
{\iintvargsaux#1} %
\NewDocumentCommand{\iintvargsaux}{mm} {#1\mkern1.5mu,\mkern1.5mu#2}
\def\intv#1[#2..#3]{\mathopen{#1[}#2\mathrel{{.}\,{.}}\nobreak#3\mathclose{#1]}}
\usepackage[font=small,labelfont=bf,
   justification=justified,
   format=plain]{caption}

\usepackage{mleftright}

\theoremstyle{plain}
\newtheorem{theorem}{Theorem}

\theoremstyle{definition}
\newtheorem{definition}[theorem]{Definition}

\theoremstyle{plain}
\newtheorem{proposition}[theorem]{Proposition}

\theoremstyle{plain}
\newtheorem{corollary}[theorem]{Corollary}

\theoremstyle{plain}
\newtheorem{lemma}[theorem]{Lemma}

\theoremstyle{definition}
\newtheorem{example}[theorem]{Example}

\theoremstyle{definition}
\newtheorem{remark}[theorem]{Remark}

\theoremstyle{plain}

\theoremstyle{plain}

\theoremstyle{plain}

\theoremstyle{plain}
\newtheorem{conjecture}[theorem]{Conjecture}

\usepackage{authblk}

\title{Multi-Trek Separation in Linear Structural Equation Models}
\date{}
\author[*]{Elina Robeva}
\author[$\dag$]{Jean-Baptiste Seby}
\affil[*]{\textit{The University of British Columbia}}\affil[$\dag$]{\textit{Massachusetts Institute of Technology}}
\begin{document}
\maketitle
\vspace{-10mm}
\begin{abstract}
\justify
Building on the theory of causal discovery from observational data, we study interactions between 
multiple (sets of) random variables in a linear structural equation model with non-Gaussian error terms.
We give a correspondence between structure in the higher order cumulants and combinatorial structure in the causal graph. It has previously been shown that low rank of the covariance matrix corresponds to trek separation in the graph. Generalizing this criterion to multiple sets of vertices, we characterize when determinants of subtensors of the higher order cumulant tensors vanish. This criterion applies when  hidden variables are present as well. 
For instance, it allows us to identify the presence of a hidden common cause of $k$ of the observed variables.
\end{abstract}

\unmarkedfntext{\noindent \hspace{-0.33cm} Keywords: Graphical models, Independent Component Analysis, Trek Separation, High-order cumulants

\hspace{0.0001cm} {{MSC2020}} Subject Classification: {62R01, 62H22}}

\normalsize

\section{Introduction}\label{section_1}
\justify
Although randomized experiments are the most commonly used method for causal inference, they are sometimes not feasible for practical or ethical reasons. Because of these constraints, scientists often need to learn the structure of the graph underlying the relationships between variables based on purely observational data. 
Suppose that $G = (V, \mathcal D)$ is a directed acyclic graph (or DAG) with vertex set $V = \{1,...,p\}$ and edge set $\mathcal D \subseteq V \times V$. 
The graph $G$ gives rise to a {\em linear structural equation model} (LSEM), which consists of joint distributions of a random vector $X = (X_1,\ldots, X_p)$ in which the variable $X_i$ associated to vertex $i\in V$ is a  linear function of $X_j$, where $j$ varies over the parent set pa$(i)$ of $i$ (i.e., all vertices  $j\in V$ such that $j\to i\in \mathcal D$), and a noise term $\varepsilon_i$,

\begin{equation}\label{linear_structural_equation_model}
    X_i = \sum_{j\in\text{pa}(i)} \lambda_{ji}X_j + \varepsilon_i, \ i \in V.
\end{equation}

If no hidden variables are present, we assume that the noise terms $\varepsilon_i$ are mutually independent. To encode the presence of hidden variables, we allow dependencies between the $\varepsilon_i$ variables, and graphically we depict this via {\em multi-directed edges} (see Figure \ref{fig_1a}). These encode a hidden common cause of a few of the observed variables. We represent this more complicated hidden structure via a {\em mixed graph} $G = (V, \mathcal D, \mathcal H)$, where $\mathcal H$ is the set of multi-directed (hyper)edges (see Section \ref{hidden variables}).


\begin{figure}[H]\label{example_introduction}
	\centering
	\begin{subfigure}[b]{0.2\linewidth}
		\includegraphics[scale = 0.45]{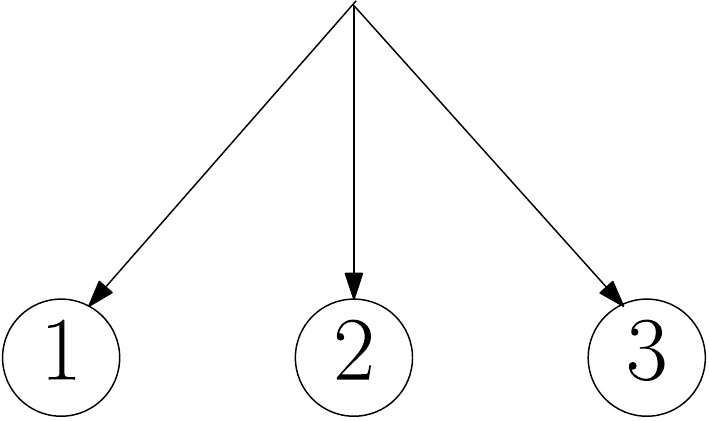}
	\caption{\label{fig_1a}}
	\end{subfigure}\hspace{0.05\textwidth}%
	\begin{subfigure}[b]{0.2\linewidth}
	\includegraphics[scale = 0.35]{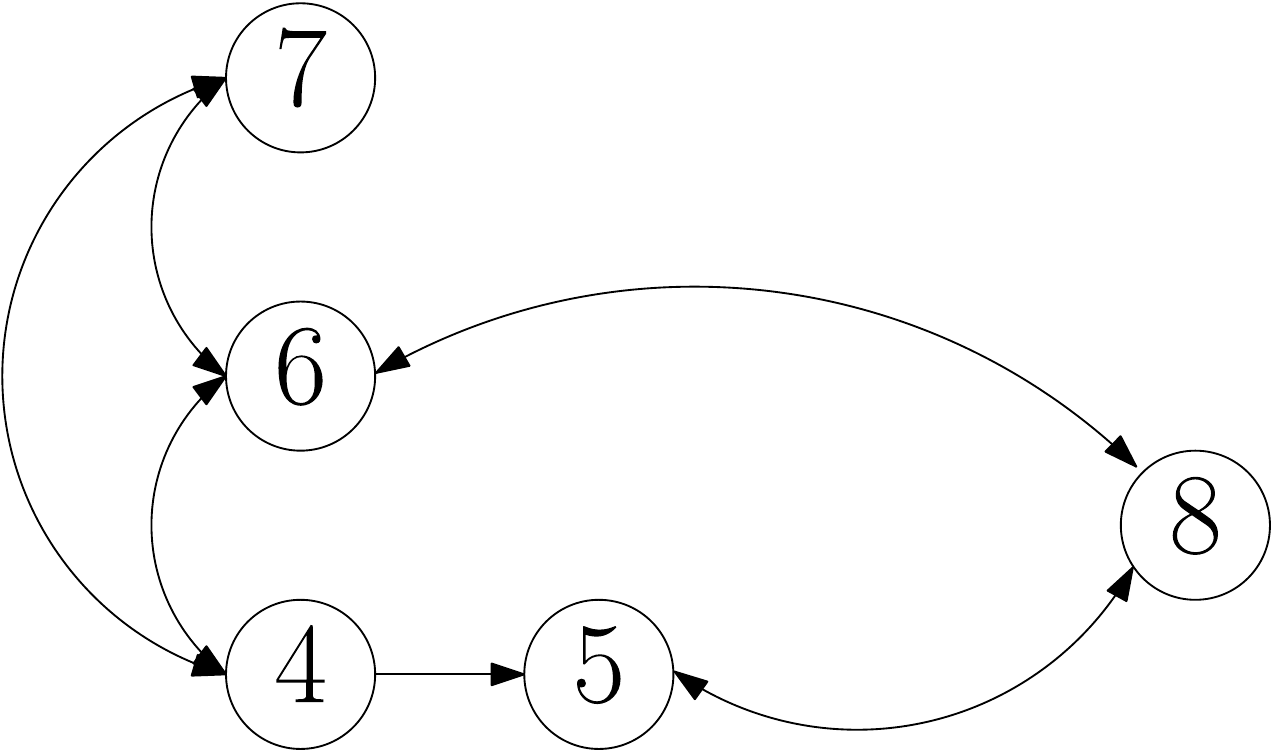}
	\caption{\label{fig:figure_DAG_intro_hidden_variables}}
	\end{subfigure}\hspace{0.1\textwidth}
		\begin{subfigure}[b]{0.2\linewidth}
		\includegraphics[scale = 0.2]{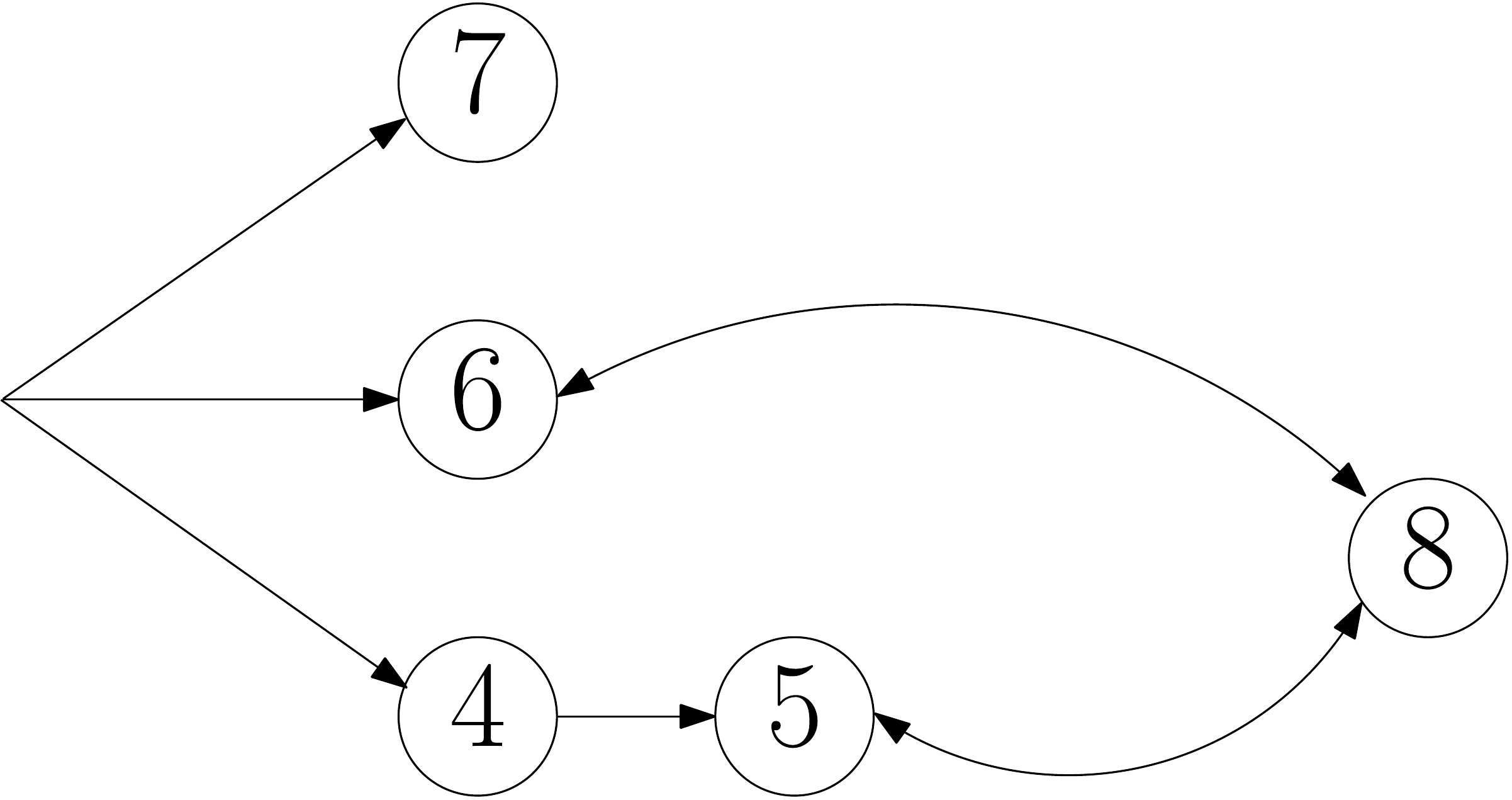}
	\caption{\label{fig:figure_DAG_intro_hidden_multi_edge}}
	\end{subfigure}
	\caption{\textit{\textbf{(a)} A multi-directed edge between nodes 1,2 and 3; \textbf{(b)} Mixed graph $G_1$; \textbf{(c)} Mixed graph $G_2$ }}
\end{figure}

When the noise terms $\varepsilon_i$ are Gaussian, then so are the $X_i$ variables. 
In this setting, the linear structural equation model given by a graph $G$ corresponds to the set of covariance matrices $\mathcal M^{(2)}(G)$ of a Gaussian distribution consistent with the graph $G$ \cite{lauritzen1996}. This is precisely the set of covariance matrices that possess a certain {\em parametrization} arising from the structure of $G$. Furthermore, bidirected edges suffice to parametrize the model $
\mathcal M^{(2)}(G)$ when hidden variables are present. For example, the mixed graphs $G_1$ and $G_2$ in the figure above give rise to the same model in Zariski closure $\overline{\mathcal M^{(2)}(G_1)} = \overline{\mathcal M^{(2)}(G_2)}$, because the two models have the same parametrization (via the trek rule~\cite{Sullivant2008}).
When the variables are non-Gaussian, we can depict the model using covariances as well as higher-order moments/cumulants of the random vector $X$. We denote by $\mathcal M^{(k)}(G)$ the set of cumulants of order $k$ consistent with the graph $G$, and by $\mathcal M^{(\leq k)}(G)$ the set of cumulants of order $i$ for $2\leq i \leq k$. These sets can also be parametrized using the graph (Definition \ref{parametrization}), and provide a more refined description of the graph structure. For instance, the two graphs above give rise to different models $\overline{\mathcal M^{(\leq3)}(G_1)}\neq \overline{\mathcal M^{(\leq3)}(G_2)}$. 

A parametrization of the model, however, may not always be  sufficient. Statistical problems like  model selection, model equivalence, and constraint based statistical inference often require an implicit description of the model in terms of (polynomial) equations which can be read off from the graph $G$, e.g., via a  combinatorial criterion. 

When $G$ is a DAG and the variables are Gaussian, the implicit description of the model $\mathcal M^{(2)}(G)$ is given by the vanishing of specific subdeterminants of the covariance matrix which can be read off from the graph via \emph{d-separation} and the more general \emph{trek-separation} criteria \cite{Sullivant2008}. In fact, the \emph{trek-separation} criterion helps describe the vanishing of all subdeterminants of the covariance matrix in any (not necessarily Gaussian) LSEM.
It turns out that covariance information is only sufficient to identify the graph up to {\em Markov equivalence}. That is, if two graphs $G$ and $G'$ give rise to the same contidional independence relations, then they produce the same sets of covariance matrices $\mathcal M^{(2)}(G) = \mathcal M^{(2)}(G')
$~\cite{Roozbehani}. Therefore, when the graph $G$ is a DAG and the variables are Gaussian, we can only recover $G$ up to Markov equivalence given observational data. Finding an implicit description of $\mathcal M^{(2)}(G)$ in the presence of hidden variables  is more challenging, although there has been promising recent progress. In particular, the authors of~\cite{Yao2019} prove that the minimal generators for the vanishing ideal $\mathcal{I}(G)$ containing all the constraints for a Gaussian Acyclic Directed Mixed Graph $G$ are in one-to-one correspondence with the pairs of non-adjacent vertices in the graph, and provide an algorithm to find all these generators. The paper \cite{Drton2018} points out that the generators of $\mathcal I(G)$ are given by nested determinants. 
\noindent

When the variables are non-Gaussian, the graph $G$ can be recovered uniquely from observational data. In particular, \cite{Shimizu2006} use Independent Component Analysis (ICA)~\cite{Comon1994} to estimate the graph structure via the Linear non-Gaussian Acyclic Model (LiNGAM). This framework and its derived versions DirectLiNGAM and PairwiseLiNGAM \cite{Hyvaerinen2013, Shimizu2011} make it possible to distinguish graphs within Markov equivalence classes. Furthermore, \cite{Wang2018} provide an algorithm that extends causal discovery of the causal structure in the high-dimensional setting based on higher-order moments, under a maximum in-degree condition.  



In this paper, we also work under the framework of a non-Gaussian LSEM. Building on the {trek rule}~\cite{Sullivant2008}, we define the {\em multi-trek rule} (Proposition \ref{prop_equation_entry_C}) which gives a polynomial parametrization of the higher-order moments/cumulants, and enables us to study LSEMs via their higher-order cumulant representation $\mathcal M^{(k)}(G)$ from the perspective of algebraic statistics~\cite{Sullivant2018} which has so far only been used for Gaussian and discrete graphical models. 
By analogy with the vanishing of subdeterminants in the covariance matrix, we give a necessary and sufficient combinatorial criterion, called \emph{multi-trek separation}, for the vanishing of subdeterminants of the tensor $\mathcal{C}^{(k)}$ of $k$-th order cumulants (Theorem~\ref{main_theorem}), which extends to the hidden variable case (Theorem~\ref{main_theorem_hidden_variables}). Our multi-trek separation criterion, for example, enables us to identify the presence of a hidden common cause of multiple observed variables.

\begin{figure}[H]
	\centering
	\begin{subfigure}[b]{0.17\linewidth}
		\includegraphics[scale=0.5]{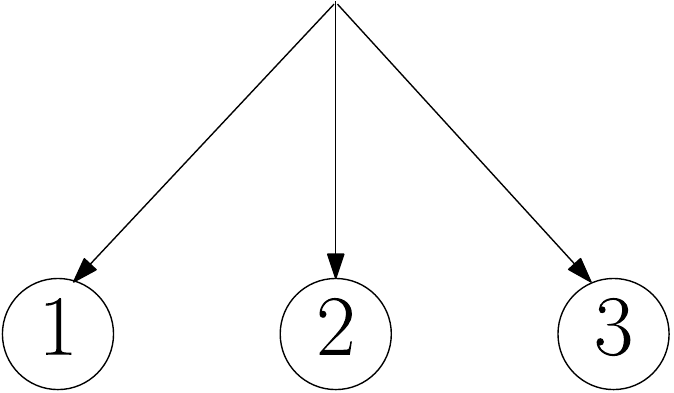}
	\caption{\label{fig:fig_intro1}}
	\end{subfigure}\hspace{0.1\textwidth}%
	\begin{subfigure}[b]{0.2\linewidth}
		\includegraphics[scale=0.46]{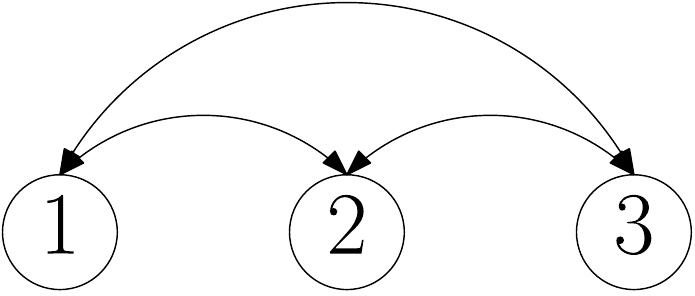}
	\caption{\label{fig:fig_intro2}}
	\end{subfigure}
	\caption{\textit{By Theorem \ref{main_theorem_hidden_variables}, (a)  $\mathcal{C}^{(3)}_{123} \neq 0$ (b) $\mathcal{C}^{(3)}_{123} = 0$}}
\end{figure}

The rest of the paper is organized as follows.  In Section \ref{section_2}, we define linear structural equation models (LSEMs) and their cumulant tensors. In Section \ref{section_3}, we introduce the notion of a multi-trek and we state our main theorem for DAG models that establishes a combinatorial criterion for the vanishing of subdeterminants of the $k^{th}$-order cumulant tensor. In Section \ref{section_4} we consider the case of hidden variables. Graphically we encode the presence of such variables via multi-directed edges, and we show that our results generalize to this case.  In Section \ref{section_5}, we conjecture that our multi-trek criterion is also equivalent to the vanishing of subdeterminants of higher-order moment (rather than cumulant) tensors. In section \ref{section_6}, we conclude and discuss directions for further research.

\section{Background}\label{section_2}
In this section we provide the necessary background on linear structural equation models, and their higher-order cumulants.
\subsection{Linear structural equation models}
\noindent
Let $G = (V, \mathcal D)$ be a directed acyclic graph (DAG) with finite vertex set $V = \{1,\ldots, p\}$ and edge set $\mathcal D \subseteq V \times V$.  Here acyclic means that there are no directed cycles, i.e., no sequences of the form $i_0 \to i_1\to\cdots\to i_s = i_0$, where $i_j\to i _{j+1}\in\mathcal D$. The
edge set is always assumed to be free of self-loops, so $i\to i \not\in \mathcal D$ for all $i \in V$. For each vertex
$i$, define its set of parents as pa$(i) = \{j \in V : j \rightarrow i \in \mathcal D\}$. The graph $G$ induces a statistical
model, called a {\em linear structural equation model}, for the joint distribution of a collection of random variables ($X_i
, i \in V $), indexed by
the graph’s vertices. The model hypothesizes that each variable is a linear function of the parent variables and a noise term $\varepsilon_i$:
\begin{equation}\label{linear_structural_equation_model_2}
    X_i = \lambda_{0i} + \sum_{j\in\text{pa}(i)} \lambda_{ji}X_j + \varepsilon_i, \ i \in V.
\end{equation}
The $\varepsilon_i$ variables for
$i \in V$, are independent and centered. The coefficients $\lambda_{0i}$ and $\lambda_{ji}$ are unknown real parameters that are assumed to be such that the
system \eqref{linear_structural_equation_model_2} admits a unique solution $X = (X_i
: i \in V )$. Typically termed a system
of structural equations, \eqref{linear_structural_equation_model_2} specifies cause-effect relations whose straightforward interpretability explains the wide-spread use of the models \cite{PeterSpirtes2000,Pearl2009}.

The random vector $X$ that solves the system \eqref{linear_structural_equation_model_2} may have an arbitrary mean depending on the choice of parameters $\lambda_{0i}$. Since the mean can easily be learned from data, and we are mainly concerned with learning the underlying graph structure, we disregard the offsets $\lambda_{0i}$, and the system~\eqref{linear_structural_equation_model_2} becomes
\begin{equation}\label{general equation_LSEM}
X = \Lambda^T X + \varepsilon.
\end{equation} 
Here $\Lambda = (\lambda_{ij}) \in \mathbb{R}^{\mathcal D}$, and $\mathbb{R}^{\mathcal{D}}$ is the set of $V \times V$ matrices $\Lambda$ with support $\mathcal{D}$, $$\mathbb{R}^\mathcal{D}  = \{\Lambda \in \mathbb{R}^{V \times V} : \lambda_{ij} = 0 \text{ if } i \rightarrow j \notin \mathcal{D}\}.$$ 
Since $G$ is acyclic, the matrix $I-\Lambda$ is always invertible and the solution of the system \eqref{general equation_LSEM} is:
\begin{equation}\label{equation_system}
    X = (I - \Lambda)^{-T}\varepsilon.
\end{equation} 

\subsection{Cumulants of linear structural equation models}

Recall that the $k$-th {\em cumulant tensor} of a random vector $Z=(Z_1,\ldots, Z_p)$ is the $p\times\cdots\times p$ ($k$ times) table with entry at position ($i_1,\ldots, i_k$) given by
$$\text{cum}(Z_{i_1}, \ldots, Z_{i_k})= \sum_{(A_1,\ldots, A_L)} (-1)^{L-1}(L-1)!\, \mathbb E\left[\prod_{j\in A_1} Z_j\right]
\mathbb E\left[\prod_{j\in A_2} Z_j\right]\cdots
\mathbb E\left[\prod_{j\in A_L} Z_j\right],$$
where the sum is taken over all partitions $(A_1,\ldots, A_L)$ of the set $\{i_1,\ldots, i_k\}$. If each of the variables $Z_i$ is centered, i.e., has mean 0, then we can restrict to summing over partitions for which each $A_i$ has size at least~2. 
For example, the first four cumulants are given as follows: 
    $$ \text{cum}(Z_{i}) =  \mathbb{E}[Z_{i}] = 0, \quad\quad \text{cum}(Z_{i_1}, Z_{i_2}) =  \mathbb{E}[Z_{i_1}Z_{i_2}], \quad\quad \text{cum}(Z_{i_1}, Z_{i_2}, Z_{i_3}) =  \mathbb{E}[Z_{i_1}Z_{i_2}Z_{i_3}], \quad\text{and}$$
    $$\text{cum}(Z_{i_1}, Z_{i_2}, Z_{i_3}, Z_{i_4}) = \mathbb{E}[Z_{i_1}Z_{i_2}Z_{i_3}Z_{i_4}] - \mathbb{E}[Z_{i_1}Z_{i_2}]\mathbb{E}[Z_{i_3}Z_{i_4}] - \mathbb{E}[Z_{i_1}Z_{i_3}]\mathbb{E}[Z_{i_2}Z_{i_4}] - \mathbb{E}[Z_{i_1}Z_{i_4}]\mathbb{E}[Z_{i_2}Z_{i_3}].$$

Now, let $k\geq 2$, and let $\mathcal E^{(k)}$ and $\mathcal C^{(k)}$ be the $k$-th order cumulant tensors of the random vectors $\varepsilon$ and $X$, respectively. The linear structural equation model, and, particularly, the expression~\eqref{equation_system} yield the following relationship between $\mathcal C^{(k)}$ and $\mathcal E^{(k)}$.

\begin{lemma}[{\cite[Chapter~5]{Comon2010}}]\label{factorization_cumulant_tensor}
The tensor $\mathcal C^{(k)}$ of $k$-th order cumulants of $X$ equals
\begin{equation}\label{tucker_decomposition_C}
   \mathcal C^{(k)} = \mathcal E^{(k)}\bullet (I-\Lambda)^{-k}, 
\end{equation}

\noindent
where $ \mathcal E^{(k)}\bullet (I-\Lambda)^{-k}$ denotes the {\em Tucker product} of the order-$k$ tensor $\mathcal E^{(k)}$ and the matrix $(I-\Lambda)^{-1}$ along each of its $k$ dimensions. In other words,
$$\left(\mathcal E^{(k)}\bullet (I-\Lambda)^{-k}\right)_{i_1,\ldots, i_k} = \sum_{j_1,\ldots, j_k} \mathcal E^{(k)}_{j_1,\ldots, j_k} ((I-\Lambda)^{-1})_{j_1,i_1} \cdots ((I-\Lambda)^{-1})_{j_k,i_k}.$$
\end{lemma}

\noindent
When the entries of the noise vector $\varepsilon$ are mutually independent, its cumulants $\mathcal E^{(k)}$ are  {\em diagonal} tensors. 

\noindent
\begin{lemma}[{\cite[Chapter~5]{Comon2010}}]\label{lem:diagonal_cumulants}
If the variables $Z_1,\ldots, Z_p$ are independent, then the $k$-th order cumulant tensor of $Z=(Z_1,\ldots, Z_p)$ is diagonal, i.e., the entry at position $(i_1,\ldots, i_k)$ is 0 unless $i_1, \ldots, i_k$ are all equal.
\end{lemma}

\begin{remark} (a)  We originally considered the moments of $X$ and $\varepsilon$, instead of their cumulants. Lemma~\ref{factorization_cumulant_tensor} also applies to the factorization of the moments, however, the $k$-th moment tensors of $\varepsilon$ for $k\geq 4$ are no longer diagonal, i.e., Lemma~\ref{lem:diagonal_cumulants} does not apply. We give further details of this study in Section \ref{section_6}.
\vspace{-0.15cm}

(b) Lemma~\ref{lem:diagonal_cumulants} is the reason cumulants are widely used for Independent Component Analysis (ICA) \cite{Comon2010}. Indeed, finding the CP-decomposition \cite{Comon2010,Lathauwer2006} of the cumulant tensor $\mathcal C^{(k)}$ can recover the matrix $(I-\Lambda)^{-1}$. For an extended study of cumulants in ICA, we refer the reader to~\cite{Comon2010}.
\end{remark}
Lemmas~\ref{factorization_cumulant_tensor} and \ref{lem:diagonal_cumulants} provide a means to {\em parametrize} the set of all cumulant tensors of the distributions in a given graphical model.

\begin{definition}\label{parametrization}
Let $G = (V, \mathcal D)$ be a DAG, and let $k\geq 2$ be an integer. The set 
$$\mathcal M^{(2)}(G) = \{(I-\Lambda)^{-T} \mathcal E^{(2)} (I-\Lambda)^{-1} : \Lambda\in\mathbb R^{\mathcal D}, \, \mathcal E^{(2)}\succeq 0 \,\text{ diagonal}\}$$
consists of all covariance matrices of distributions in the graphical model given by $G$. For $k\geq 3$, define
$$\mathcal M^{(k)}(G) = \{\mathcal E^{(k)}\bullet (I-\Lambda)^{-k} : \Lambda\in\mathbb R^{\mathcal D},\,\mathcal E^{(k)} \text{ diagonal}\}$$
to be the set of cumulants of order $k$ consistent with the graph $G$.
And finally, let
$$\mathcal M^{(\leq k)}(G) = \mathcal M^{(2)}(G)\times\cdots\times\mathcal M^{(k)}(G)$$
be the set all cumulants up to order $k$ of a distribution in the graphical model given by $G$.
\end{definition}

When the error terms $\varepsilon_i$ are Gaussian, the random vector $X$ also follows a Gaussian distribution and all of its cumulants of order $k\geq 3$ equal 0. Therefore, the model equals $\mathcal M^{(2)}(G)$. In this case, different DAGs $G_1$ and $G_2$ that lie in the same {\em Markov equivalence class} can give rise to the same models $\mathcal M^{(2)}(G_1) = \mathcal M^{(2)}(G_2)$. An implicit description of $\mathcal M^{(2)}(G)$ is well-known when $G$ is a DAG and is given by  conditional independences ~\cite{Garcia2005,Kiiveri1984}. These correspond to \textit{d-separations} in the graph $G$ \cite{Pearl1986}.

When the noise terms $\varepsilon_i$ are not Gaussian, the higher-order moments will not necessarily vanish, and we can obtain more information about the distribution from them. It turns out that $\mathcal M^{(\leq k)}(G)$ uniquely identifies the DAG $G$~\cite{Shimizu2006, Shimizu2011, Wang2018}. An implicit description of $\mathcal M^{(\leq k)}(G)$ is not known completely, although~\cite{ Wang2018} discover enough of the defining equations to identify the DAG $G$.

\noindent



\section{Multi-trek separation for directed acyclic graphs}\label{section_3}
In this section we present our main result, a particular type of constraint on the model $\mathcal M^{(k)}(G)$ that corresponds to a combinatorial criterion in the graph $G$. Given data, one can check whether the constraint holds for the sample cumulant tensors in order to obtain information about the structure of the unknown DAG $G$. Our result generalizes to the case of hidden variables as shown in Section \ref{section_4}.

\subsection{Multi-treks and the multi-trek rule}
We begin by generalizing the notion of a {\em trek}~\cite{Sullivant2008}.
\begin{definition}
\noindent
A $k$-trek in a DAG $G$ between $k$ nodes $v_1,\ldots, v_k$ is an ordered collection of $k$ directed paths $(P_{1}, ..., P_{k})$, where $P_{i}$ has sink $v_i$, and $P_{1},..., P_{k}$ have the same source vertex, called the {\em top} of the $k$-trek and denoted by $top(P_{1},..., P_{k})$.
\end{definition}
\noindent
Note that a 2-trek is exactly the same as the usual notion of a trek~\cite{Sullivant2008}.

\begin{figure}[H]
\centering
     \includegraphics[width=0.25\textwidth]{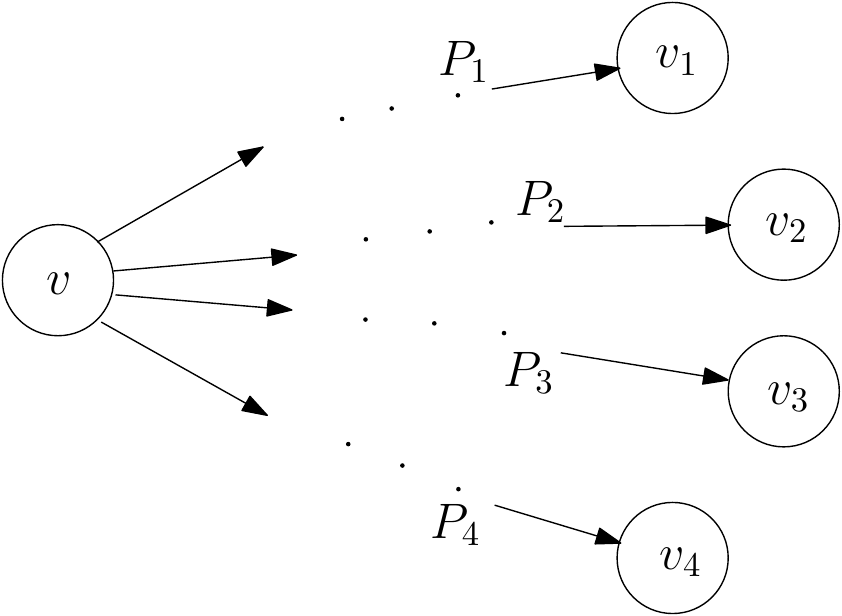}\small(a)
\hspace{1.8cm}	\includegraphics[width=0.32\textwidth]{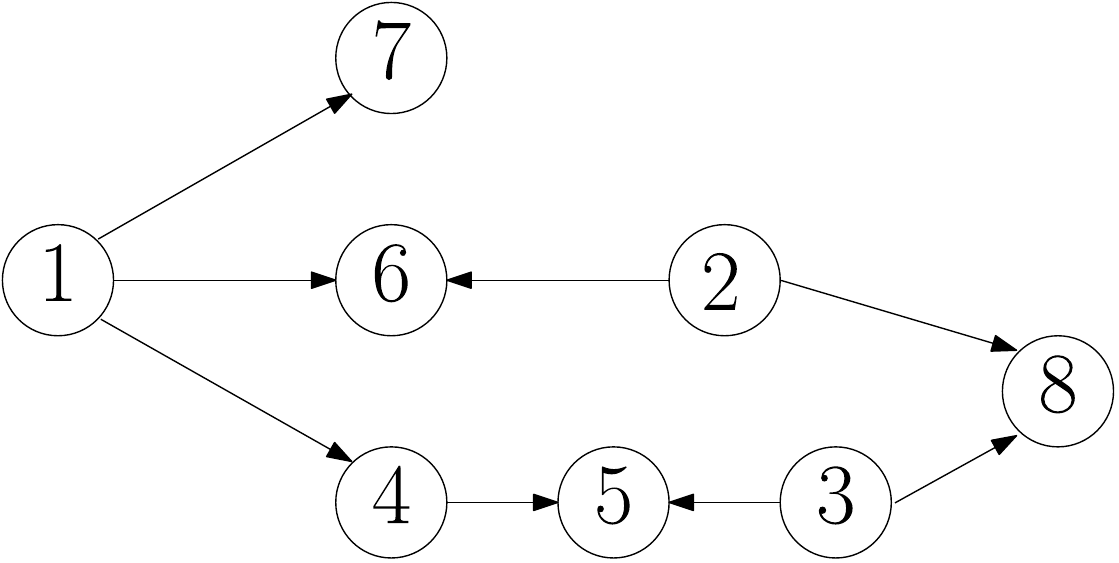}\small(b)
	\caption{(a). An example of a 4-trek; (b) Example DAG.}\label{fig:2}
\end{figure}

\begin{example}
In the directed acyclic graph from Figure~\ref{fig:2}b,
$$(2\to 6,\, 2\to 8) \text{ is a 2-trek between 6 and 8 with top 2};$$
$$(1\to7,\, 1\to6,\, 1\to4\to 5) \text{ is a 3-trek between 7, 6, and 5 with top 1};$$
$$(1\to 7,\, 1\to 6,\, 1,\, 1\to 6) \text{ is a 4-trek between 7, 6, 1, and 6 with top 1}.$$

\end{example}
\noindent
Note that the paths $P_i$ can consist of a single vertex.

\noindent
We now generalize the {\em trek rule}, which originates in the work of~\cite{Wright1921,Wright1934} and relies on the observation that $(I-\Lambda)^{-1} = I + \Lambda + \Lambda^2 + \cdots$. Since the graph $G$ is acyclic, this formal sum is finite, and each of the entries of $(I-\Lambda)^{-1}$ can be expressed as
\begin{equation}\label{entries_I-Lambda}
(I - \Lambda)^{-1}_{ji} = \sum_{P \in \mathcal{P}(j,i)} \lambda^P 
\end{equation}
where $\mathcal{P}(j,i)$ is the set of all directed paths from $j$ to $i$, and $\lambda^P$ is the product of the coefficients $\lambda_{uv}$ along the edges $u\to v$ on a directed path $P$ (e.g., if $P$ is the path $1 \rightarrow 2 \rightarrow 4$, then $\lambda^P$ is $\lambda_{12}\lambda_{24}$). Note that, since we only consider acyclic graphs, for any $i\in V$, $\mathcal P(i,i)$ consists of the trivial path from $i$ to $i$ consisting of just the vertex $i$, and the monomial for this path is equal to $1$.

\begin{proposition}[The multi-trek rule]\label{prop_equation_entry_C}
When the entries of the noise vector $\varepsilon$ are independent, the entries of the $k$-th order cumulant tensor $\mathcal C^{(k)}$ of $X$ can be expressed as a sum over {\em $k$-trek monomials},
\begin{equation}\label{eq:equation_entry_C}
\mathcal{C}^{(k)}_{i_1,...,i_k} = \sum_{(P_1,...,P_k) \in \mathcal{T}(i_1,...,i_k)} {\mathcal E^{(k)}_{top(P_{1},..., P_{k}),\ldots, top(P_{1},..., P_{k})} \lambda^{P_{1}}} ... \ \lambda^{P_{k}},
\end{equation}
where $\mathcal{T}{(i_1,...,i_k)}$ is the set of all $k$-treks between $i_1,...,i_k$.
\end{proposition}

\begin{proof}
\noindent
By Lemma~\ref{factorization_cumulant_tensor}, we can express the $k$-th cumulant tensor of $X$ via the Tucker decomposition $\mathcal{C}^{(k)} = 
\mathcal{E}^{(k)}\bullet(I - \Lambda)^{-k}$. Furthermore, by Lemma~\ref{lem:diagonal_cumulants}, the $k$-th cumulant tensor $\mathcal E^{(k)}$ of $\varepsilon$ is diagonal. Therefore,
\begin{align*}
\mathcal C^{(k)}_{i_1,\ldots, i_k} &= \sum_{j_1,\ldots, j_k}\mathcal E^{(k)}_{j_1,\ldots, j_k} ((I-\Lambda)^{(-1)})_{j_1,i_1}\cdots ((I-\Lambda)^{(-1)})_{j_k,i_k}\\
&= \sum_{j}\mathcal E^{(k)}_{j,\ldots, j} ((I-\Lambda)^{(-1)})_{j,i_1}\cdots ((I-\Lambda)^{(-1)})_{j,i_k}.
\end{align*}
Using equation (\ref{entries_I-Lambda}), we obtain
$$\mathcal C^{(k)}_{i_1,\ldots, i_k} = \sum_j\mathcal E^{(k)}_{j,\ldots, j}\left(\sum_{P_1\in\mathcal P(j,i_1)}\lambda^{P_1}\right)\cdots\left(\sum_{P_k\in\mathcal P(j,i_k)}\lambda^{P_k}\right),$$
which yields the result.
\end{proof}


\begin{example}
Consider the DAG from Figure~\ref{fig:2}b. According to the multi-trek rule, we have, for instance, the following relationships between the cumulants $\mathcal C^{(k)}$ and $\mathcal E^{(k)}$ of $X$ and $\varepsilon$.
$$\mathcal C^{(2)}_{4,5} = \mathcal E^{(2)}_{4,4}\lambda_{45}+\mathcal E^{(2)}_{1,1}\lambda_{14}^2\lambda_{45}, \quad\quad\quad\mathcal C^{(3)}_{5,6,7} = \mathcal E^{(3)}_{1,1,1}\lambda_{14}\lambda_{45}\lambda_{16}\lambda_{17}, \quad\quad\quad \mathcal C^{(3)}_{5,6,8} = 0.$$
These follow because there are two 2-treks between $4$ and 5: $(4, 4\to5)$ and $(1\to4, 1\to4\to5)$. There is one 3-trek between 5, 6, and 7: $(1\to4\to5, 1\to6, 1\to7)$. There are no $3$-treks between 5, 6, and 8.

\end{example}

\subsection{Multi-trek systems and determinants of higher-order cumulants}
We now generalize the multi-trek rule, giving an expression of the {\em determinants} of subtensors of $\mathcal C^{(k)}$ in terms of {\em multi-trek systems}. We first recall the notion of a combinatorial hyperdeterminant~\cite{Cayley} of a tensor, which we simply refer to as determinant throughout the rest of the paper
\begin{definition}\label{definition_determinant_tensor}
\noindent
Let $T$ be an order-$k$ $n \times ... \times n$ tensor. Then, its determinant is
\begin{equation}
\text{det}(T) = \sum_{\sigma_2,..., \sigma_{k}\in\mathfrak{S}(n)} \text{sign}(\sigma_2) \cdots \text{sign}(\sigma_{k}) \prod_{i=1}^{n} T_{i,\sigma_2(i),...,\sigma_{k}(i)},
\end{equation}
where $\mathfrak{S}(n)$ is the set of permutations of the set $\{1,\ldots, n\}$.
\end{definition}
\noindent Next, we introduce the notion of a multi-trek system.

\begin{definition}
Given a collection of $k$ sets of nodes $S_1,...,S_k\subseteq V$ such that $\#S_1 =...=\#S_k = n$, a {\em $k$-trek system} ${T}$ is a collection of $n$ $k$-treks between $S_1,...,S_k$ such that the ends of $T$ on the $i$-th side equal $S_i$. We define the {\em top} of this $k$-trek system, $top({T})$, to be the union of the tops of the $k$-treks. We allow repeated elements in $top({T})$. 
A $k$-trek system $ T$ has {\em sided intersection} if there exist two $k$-treks $(P_1,\ldots, P_k)$ and $(Q_1,\ldots, Q_k)$ in $ T$ and a number $1\leq i\leq k$ so that the directed paths $P_i$ and $Q_i$ have a common vertex. We denote by $\widetilde{\mathcal T}(S_1,\ldots, S_k)$ the set of $k$-trek systems between $S_1,\ldots, S_k$ that have {\em no} sided intersection.
\end{definition}

\begin{example} In Figure 4 below, the two 3-treks between $S_1, S_2, S_3$ (respectively in full and dashed line) form a 3-trek-system between $S_1,S_2,S_3$. They have a sided intersection along the paths leading to the set~$S_1$.
\begin{figure}[H]
\centering
\includegraphics[width=0.3\linewidth]{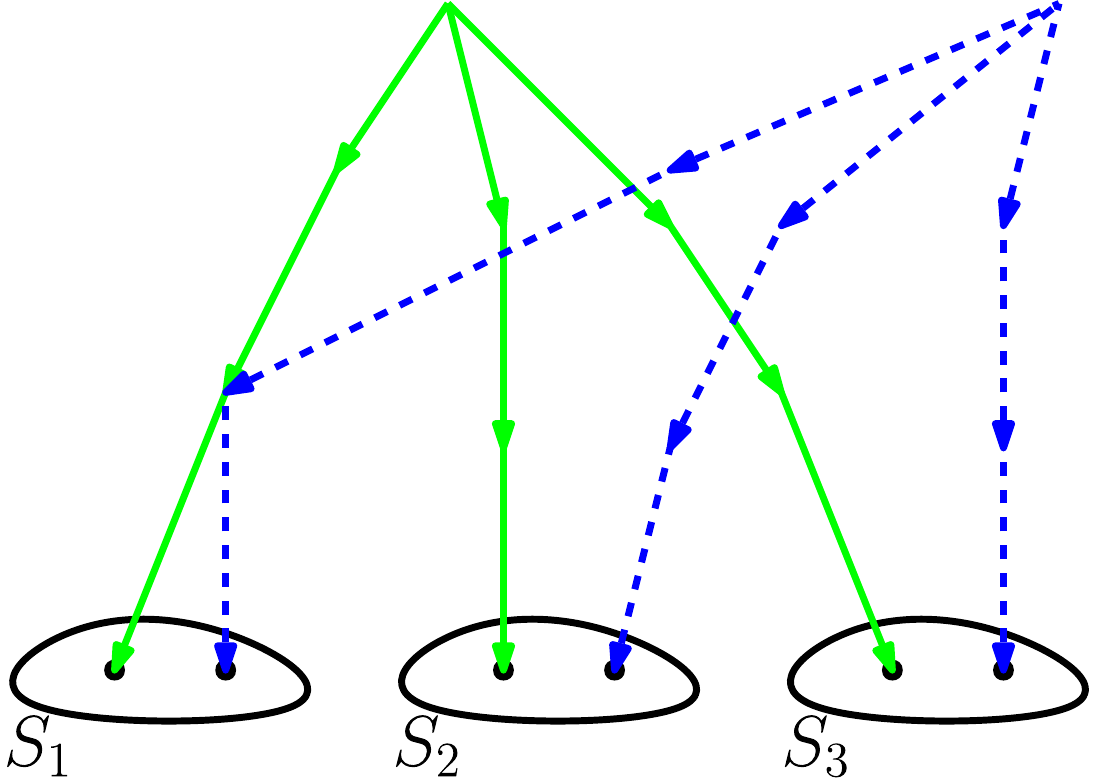}\label{sided_intersection}
\caption{A tri-trek system with sided intersection.}
\end{figure}
\end{example}

Given a collection of $k$ sets of nodes $S_1,...,S_k\subseteq V$ such that $\#S_1 =...=\#S_k = n$, and an ordering of the nodes in each set $S_i$, a $k$-trek system $T$ gives rise to a permutation of the nodes in each of $S_2,...,S_k$ (if we keep the ordering of $S_1$ fixed). More explicitly, the $j$-th $k$-trek in $T$ connects the $j$-th vertex of $S_1$  with the $\sigma_i(j)$-th vertex of $S_i$ for all $i=2,\ldots, k$, and $\sigma_2, \ldots, \sigma_k$ are the  permutations induced by $T$.  The {\em sign} of $T$ is the product of the signs of those $(k-1)$ permutations. The signs of the permutations are defined with respect to the ordering chosen initially.

\begin{example} In Figure~\ref{fig:sign_trek_system} below, $S_1=\{v_1,v_2\}$, $S_2=\{v_3,v_4\}$, $S_3=\{v_5,v_6\}$. Assume that the initial ordering is $(v_1,v_2), (v_3, v_4), (v_5, v_6)$. The trek system in Figure~\ref{fig_sign_T_a} has two 3-treks, one between $v_1,v_4, v_5$ and one between $v_2, v_3, v_6$. Therefore,  sign$(T) = (-1) \times 1 = -1$. 
The trek system in Figure~\ref{fig_sign_T_b} has two 3-treks, one between $v_1,v_4, v_6$ and one between $v_2, v_3, v_5$. Therefore,  sign$(T) =  (-1) \times (-1) =  +1$.
\begin{figure}[H]
	\centering
	\begin{subfigure}[b]{0.4\linewidth}
		\includegraphics[scale=0.4]{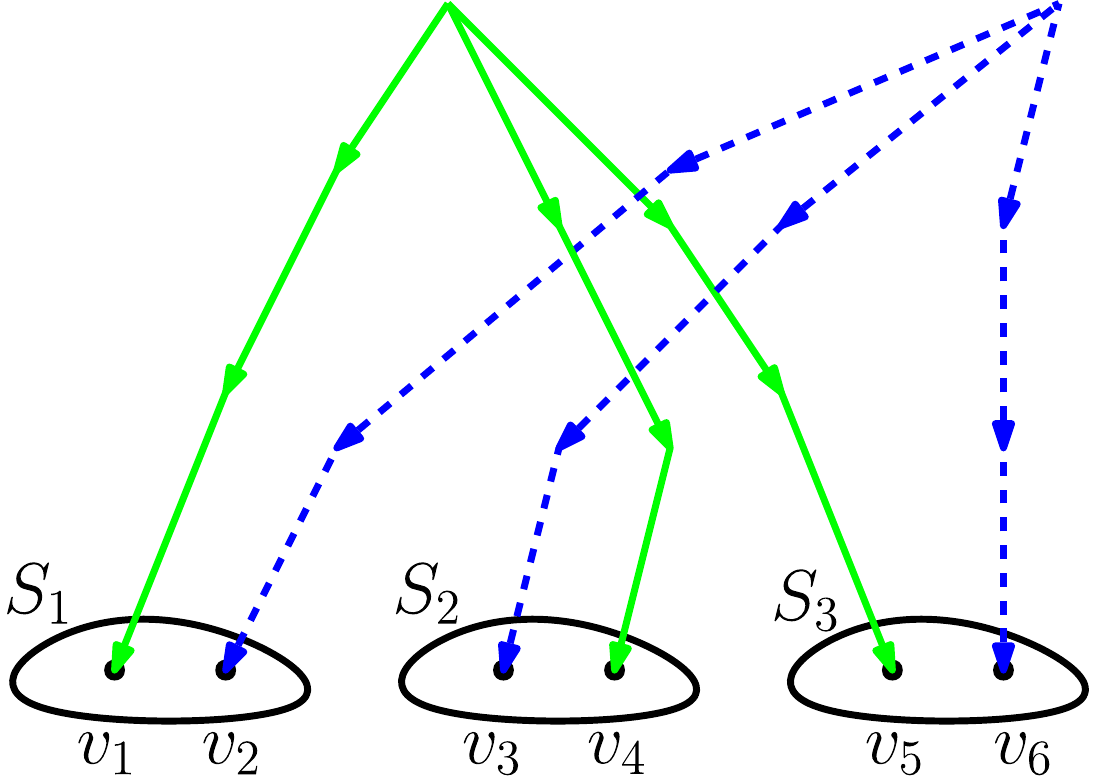}
	\caption{\label{fig_sign_T_a}}
	\end{subfigure}\hspace{0.03\textwidth}%
	\begin{subfigure}[b]{0.4\linewidth}
		\includegraphics[scale=0.4]{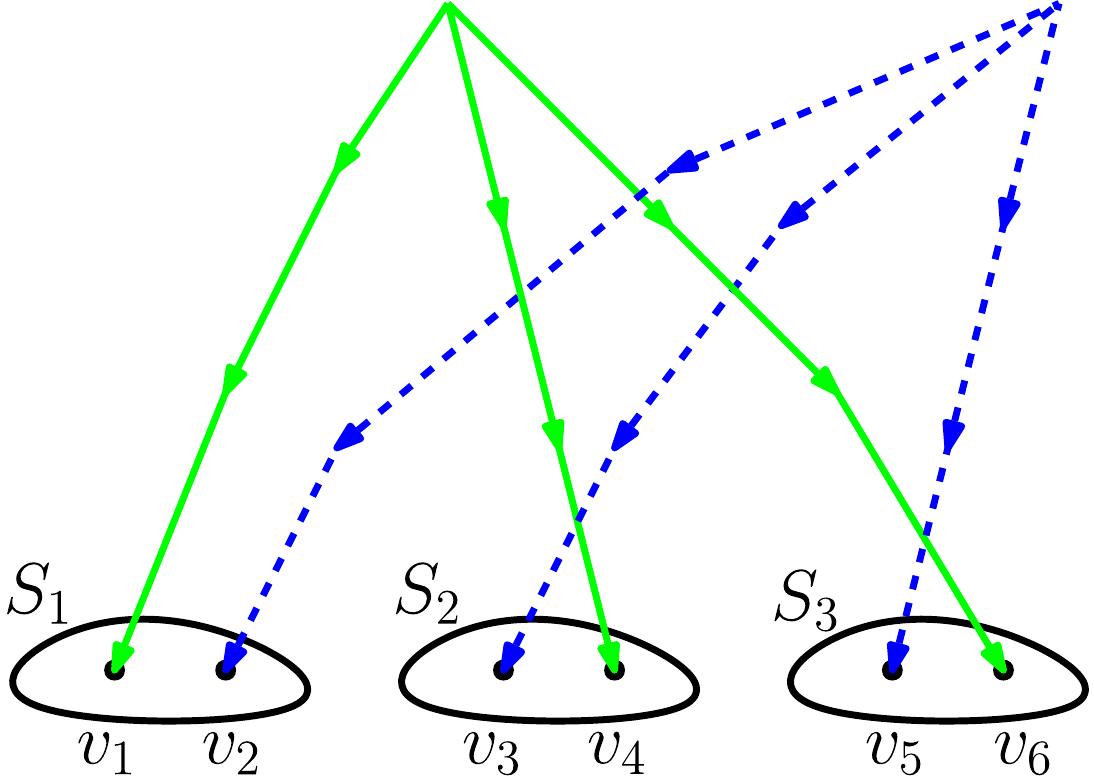}
	\caption{\label{fig_sign_T_b}}
	\end{subfigure}
	\caption{\label{fig:sign_trek_system}}
\end{figure}
\end{example}

The determinant of the subtensor of $k^{th}$ order cumulants indexed by the sets $S_1,...,S_k$ can be expressed in terms of the trek systems involving $S_1,...,S_k$.
\begin{proposition}\label{det_T_trek_system_m}
Let $S_1,\ldots, S_k\subseteq V$ be $k$ sets of nodes such that $\#S_1 =...=\#S_k = n$. Then,
\begin{align*}
    \text{det } \mathcal C^{(k)}_{S_1,\ldots, S_k} = \sum_{T\in\mathcal T(S_1,\ldots, S_k)} \text{sign(T) } m_{T},  \tag{\textasteriskcentered} \label{eq_ast}
\end{align*}
where $\mathcal T(S_1,\ldots, S_k)$ is the set of $k$-trek systems between $S_1,\ldots, S_k$, and $m_T$ is the trek-system monomial of the trek system $T=\{(P^{(1)}_1,\ldots, P^{(1)}_k), \ldots, (P^{(n)}_1,\ldots, P^{(n)}_k)\}$, defined as 
$$m_T = \prod_{i=1}^n \mathcal E^{(k)}_{top((P^{(i)}_1,\ldots, P^{(i)}_k)),\ldots, top((P^{(i)}_1,\ldots, P^{(i)}_k))}\prod_{j=1}^k\lambda^{P^{(i)}_j}.$$
In fact, the sum in (\ref{eq_ast}) can be taken over treks $T \in \widetilde{\mathcal{T}}(S_1,\ldots, S_k)$ without sided intersections, i.e.,
\begin{equation}
\text{det } \mathcal C^{(k)}_{S_1,\ldots, S_k} = \sum_{T\in \widetilde{\mathcal{T}}(S_1,\ldots, S_k)} \text{sign(T) } m_{T}.
\label{sum_over_trek_system_without_sided_intersection}
\end{equation}
\end{proposition}

\begin{proof}[Proof of Proposition \ref{det_T_trek_system_m}]
	\noindent
	From Propositon \ref{prop_equation_entry_C}, we can write
	\begin{equation}
	\begin{aligned}
	\mathcal{C}^{(k)}_{i_1,...,i_k} &= \sum_{T \in \mathcal{T}(i_1,...,i_k)} m_T,
	\end{aligned}
	\end{equation}
	where $m_T$ is the $k$-trek monomial defined by $m_T = {\mathcal{E}^{(k)}_{top(P_{1},..., P_{k})} \lambda^{P_{1}}} ....  \lambda^{P_{k}}$, where $\lambda^{P_{j}} = \prod_{k \rightarrow l \in P_{j}} \lambda_{kl}$.
	
	Assuming that $\#R_1= ...=\#R_k = \#S_1 = ... = \#S_k = n$ and using the Leibniz expansion formula for determinants, we then get:
	\begin{equation}\label{equations_sum_trek_systems}
	\begin{aligned}
	\text{det } \mathcal{C}^{(k)}_{S_1,...,S_k} \ &= \sum_{\sigma_2 \in \mathfrak{S}_{n_{2}},..., \sigma_{k} \in \mathfrak{S}_{n_{k}}} \Bigg(\sum_{\substack{T_1 \in \mathcal{T}(s_1, s_{\sigma_2(1)},
			...,s_{\sigma_{k}(1)})\\
			...\\
			...\\
			T_n \in \mathcal{T}(s_n, s_{\sigma_2(n)},...,s_{\sigma_{k}(n)})}} \text{sign}(\sigma_2) \cdots \text{sign}(\sigma_{k})\ m_{T_1}...m_{T_n} \Bigg)\\
	&= \sum_{\substack{\sigma_i \in \mathfrak{S}_{n_i}\\
			i\in \{2,\ldots,k\}}} \Bigg(\sum_{\substack{T_j \in \mathcal{T}(s_j, s_{\sigma_{2}(j)}, \ldots, s_{\sigma_{k}(j)})\\
			j \in \{1,\ldots,n\}}}
	\prod_{l = 2}^{k}\text{sign}(\sigma_l)\prod_{s = 1}^{n} m_{T_s}\Bigg)\\
	&= \sum_{T\in\mathcal T(S_1,\ldots, S_k)} \text{sign}(T) \  m_{T},
	\end{aligned}
	\end{equation}
	
	\noindent
	where $\mathfrak{S}_{n_{i}}$ is the set of permutations of the nodes in $S_{i}$, $T$ runs over all $k$-trek systems between $S_1,...,S_k$ and $\text{sign}(T) =  \text{sign}(\sigma_2) \cdots \text{sign}(\sigma_{k})$. In this expression, we have $m_{T} = \prod_{x=1}^{n} m_{T_x}$, where $m_{T_x} = \mathcal{E}^{(k)}_{top(P_{s_{1}},...,P_{s_{k}})}$ $\lambda^{P_{s_{1}}} \cdots \lambda^{P_{s_{k}}}$ is the trek monomial corresponding to the trek $T_x$.
	
	\noindent

	In order to now prove the equation  (\ref{sum_over_trek_system_without_sided_intersection}) $$\text{det } \mathcal C^{(k)}_{S_1,\ldots, S_k} =  \sum_{T\in \widetilde{\mathcal{T}}(S_1,\ldots, S_k)} \text{sign}(T) m_{T},$$ we need to show that $m_T = 0$ when $T$ is a trek-system between $S_1,...,S_k$ with a sided intersection.
	
	We first prove a tensor version of the Cauchy-Binet Theorem~\cite{Broida} for the determinant of the product $AB$, where $A$ is a tensor of order $k$ and $B$ is a matrix.
	
	\begin{lemma}\label{Cauchy-Binet theorem}
	Suppose that A is a $\underbrace{p \times n\times p \times ... \times p} _\text{k}$ tensor of order $k$ and $B$ is a $n \times p$ matrix, $I$ is a subset of $\{1,...,p\}$. Then \text{det }$(AB) = \sum_{I\subseteq [n], \#I=p} \text{det }(A_I) \text{det }(B_{I})$, where the sum is over subsets $I$ of $\{1,...,n\}$ such that $\#I = p$, and $A_I$ denotes the subtensor of $A$ given by $A_{[p], I, [p], \ldots, [p]}$, and $B_I$ denotes the submatrix of $B$ given by $B_{I, [p]}$.
\end{lemma}
The proof of Lemma~\ref{Cauchy-Binet theorem}
is presented in Appendix \ref{Appendix A}. Recall that
Lemma \ref{factorization_cumulant_tensor} gives the relationship between the $k$-th order cumulant tensors of the random vectors $\varepsilon$ and $X$,    $\mathcal C^{(k)} = \mathcal E^{(k)}\bullet (I-\Lambda)^{-k} $. We can apply the determinant operator to each side of this equation, using the tensor version of the Cauchy-Binet theorem proved above.
\noindent
\begin{lemma}\label{det_cumulant}
	\noindent
	Consider $k$ subsets of vertices $S_1, ..., S_k \subset V$ with $ \#S_1 = ... = \#S_k$. Then, the determinant of the subtensor of $k^{th}$-order cumulants can be written as:
	\begin{equation}
	\text{det }\mathcal{C}^{(k)}_{S_1,...,S_k} =  \sum\limits_{\substack{R_1,...,R_k \subseteq V, \\ \#R_i = \#S_i}}\text{det }\mathcal{E}^{(k)}_{R_1,...,R_k} \
	\text{det } (I - \Lambda)^{-1}_{R_1,S_1} \ ... \ \text{det }(I - \Lambda)^{-1}_{R_k,S_k}\label{equation_C_after_Cauchy_Binet}
	\end{equation}
\end{lemma}
\begin{proof}
	\noindent
	We apply the tensor version of the Cauchy-Binet Theorem stated in Proposition \ref{Cauchy-Binet theorem} $k$ times to the Tucker product in equation (\ref{tucker_decomposition_C}) and we obtain equation (\ref{equation_C_after_Cauchy_Binet}).
\end{proof}

	By Lemma \ref{det_cumulant}, we have:
	\begin{equation}\label{equation_T_after_Cauchy_Binet}
	\text{det }\mathcal{C}^{(k)}_{S_1,...,S_k} = \sum\limits_{R_1,...,R_k} \text{det }\mathcal{E}^{(k)}_{R_1,...,R_k} \
	\text{det} (I-\Lambda)^{-1}_{R_1,S_1} \ ... \ \text{det}(I-\Lambda)^{-1}_{R_k,S_k}
	\end{equation}

 A nice tool to prove this is given by the Gessel-Viennot-Lindstrom Lemma. We use this Lemma is again in the proof of our main Theorem \ref{main_theorem}. 
	
	\begin{lemma}[\cite{Gessel1985,Lindstroem1973}]\label{GVL_lemma}
		\noindent
		Suppose $G$ is a directed acyclic graph with vertex set $[p] = \{1,...,p\}$. Let $R$ and $S$ be subsets of $[p]$ with $\#R = \#S = n$. Then,
		$$ \text{det} \ (I -\Lambda)^{-1}_{R,S} = \sum_{P \in N(R,S)} (-1)^{P} \lambda^{P},$$
		where $N(R,S)$ is the set of all collections of nonintersecting systems of $n$ directed paths in $G$ from $R$ to $S$, and $(-1)^{P}$ is the sign of the induced permutation of elements from $R$ to $S$. In particular, $\text{det} \ (I-\Lambda)^{-1}_{R,S}$ is identically zero if and only if every system of $n$ directed paths from $R$ to $S$ has two paths which share a vertex.
	\end{lemma}
	
	Thus, equation~\eqref{equation_T_after_Cauchy_Binet} now yields 
	\begin{equation}\notag
	\text{det }\mathcal{C}^{(k)}_{S_1,...,S_k} = \sum\limits_{R_1,...,R_k} \text{det }\mathcal{E}^{(k)}_{R_1,...,R_k} \
	\left(\sum_{P_1\in N(R_1,S_1}(-1)^{P_1}\lambda^{P_1}\right)\cdots \left(\sum_{P_k\in N(R_k,S_k)}(-1)^{P_k}\lambda^{P_k}\right).
	\end{equation}
	Since $\mathcal E^{(k)}$ is diagonal, we have that $\text{det }\mathcal{E}^{(k)}_{R_1,...,R_k}$ is 0 unless $R_1=\cdots=R_k$. Therefore,
	\begin{equation}\notag
	\text{det }\mathcal{C}^{(k)}_{S_1,...,S_k} = \sum\limits_{R\subseteq V, \#R=n} \text{det }\mathcal{E}^{(k)}_{R,\ldots, R} \
	\left(\sum_{P_1\in N(R,S_1)}(-1)^{P_1}\lambda^{P_1}\right)\cdots \left(\sum_{P_k\in N(R,S_k)}(-1)^{P_k}\lambda^{P_k}\right)
	\end{equation}
	\begin{equation}\notag
	= \sum_{R\subseteq V, \#R = n} \prod_{r\in R}\mathcal E^{(k)}_{r,\ldots, r} \left(\sum_{P_1\in N(R,S_1)}(-1)^{P_1}\lambda^{P_1}\right)\cdots \left(\sum_{P_k\in N(R,S_k)}(-1)^{P_k}\lambda^{P_k}\right),
	\end{equation}
	
	which then yields the desired expression in~\eqref{sum_over_trek_system_without_sided_intersection}.
\end{proof}

\begin{example}
Consider, once again, the DAG from Figure~\ref{fig:2}b. Then, we have that
$$\text{det }\mathcal C^{(2)}_{46, 78} = \mathcal E^{(2)}_{1,1}\mathcal E^{(2)}_{2,2}\lambda_{14}\lambda_{17}\lambda_{26}\lambda_{28}.$$
since there is only one 2-trek system between $46$ and 78 without sided intersection, namely $\{(1\to4, 1\to 7),(2\to6, 2\to8)\}$. Similarly, we have
$$\text{det }\mathcal C^{(3)}_{46, 58, 78} = \mathcal E^{(3)}_{1,1,1}\mathcal E^{(3)}_{2,2,2} \lambda_{14}^2\lambda_{45}\lambda_{17}\lambda_{26}\lambda_{28}^2,$$
since there is only one 3-trek system between 46, 58, and 78, namely $\{(1\to4, 1\to4\to5, 1\to7), (2\to6, 2\to8, 2\to8)\}$.
\end{example}

\noindent

\subsection{Main result}

Our main theorem shows that the non-existence of a trek system without sided intersection between $k$ sets of vertices is equivalent to the vanishing of the corresponding subdeterminant of the $k$-th cumulant tensor~$\mathcal C^{(k)}$. 

\begin{theorem} \label{main_theorem}
\noindent
Let $G = (V, \mathcal D)$ be a DAG, and let $S_1,..., S_k$ be subsets of $V$ with $\#S_1 = ... = \#S_k$. Then,
$$ \text{det} \ \mathcal{C}^{(k)}_{S_1,..., S_k} = 0$$ 
for every $\mathcal C^{(k)}$ from $\mathcal M^{(k)}(G)$ if and only if every system of $k$-treks between $S_1,..., S_k$ has a sided intersection.
\end{theorem}
Before proving our main theorem, we need to introduce an additional intermediary result. 
\begin{lemma}\label{lemma_det_i_equal_0}
	Consider $k$ subsets of vertices $S_1, ..., S_k \subseteq V$ with $ \#S_1 = ... = \#S_k=n$. Then,
	$\text{det }\mathcal{C}^{(k)}_{S_1,...,S_k}$ is identically $0$ (as a polynomial in the entries of $\mathcal E^{(k)}$ and $\lambda$) if and only if for any set $R \subseteq V$ such that $\#R=n$, there exists $i \in \{1\ldots, k\}$  with  $\text{det }(I-\Lambda)^{-1}_{R,S_i} = 0$.
\end{lemma}
\begin{proof}
	\noindent
	Let us suppose that $\text{det }\mathcal{C}^{(k)}_{S_1,...,S_k}$ is identically 0. From equation (\ref{equation_C_after_Cauchy_Binet}) in Lemma \ref{det_cumulant}, we have:
	\begin{align*}
	\text{det }\mathcal{C}^{(k)}_{S_1,...,S_k} = \sum\limits_{\substack{R_1,...,R_k \subseteq V, \\ \#R_i = \#S_i}} \text{det } \mathcal{E}^{(k)}_{R_1,...,R_k} \text{det}(I-\Lambda)^{-1}_{R_1, S_1} \cdots \text{det}(I-\Lambda)^{-1}_{R_k, S_k},
	\end{align*}
	where the sum runs over subsets $R_1,...,R_k$ of $V$ of cardinality $\#R_1=...=R_k = \#S_1 = ... = \#S_k = n$. However, since $\mathcal{E}^{(k)}$ is a diagonal tensor, $\text{det } \mathcal{E}^{(k)}_{R_1,...,R_k} = 0$ unless $R_1 = R_2 = ... = R_k = R$. In this case, denoting $\text{det }\mathcal{E}^{(k)}_{R_1,...,R_k} = \text{det }\mathcal{E}^{(k)}_{R}$ we get:
	\begin{align}
	\text{det }\mathcal{C}^{(k)}_{S_1,...,S_k} = \sum\limits_{R\subseteq V} \text{det } \mathcal{E}^{(k)}_{R} \text{det}(I - \Lambda)^{-1}_{R, S_1} \cdots \text{det}(I-\Lambda)^{-1}_{R, S_k}
	\end{align}
	Each monomial $\text{det } \mathcal{E}^{(k)}_{R}$ appears only once, therefore for any set $R$ satisfying $\#R = \#S_1 = ... = \#S_k = n$, there exists $i \in \{1,\ldots, k\}$  such that  $\text{det}(I-\Lambda)^{-1}_{R,S_i} = 0$, which proves the if-direction.
	
	Let us now suppose that for any set $R$ satisfying $\#R = \#S_1 = ... = \#S_k = n$, there exists $i \in \{1,\ldots, k\}$  such that  $\text{det}(I-\Lambda)^{-1}_{R,S_i} = 0$. Then from the expression (\ref{equation_C_after_Cauchy_Binet}), we conclude that $\text{det } \mathcal{C}^{(k)}_{S_1,...,S_k} = 0$.
\end{proof}

We now present the proof of Theorem~\ref{main_theorem} which relies mostly on the Gessel-Viennot-Lindstrom Lemma (Lemma \ref{GVL_lemma}) and Lemma \ref{lemma_det_i_equal_0}.
\noindent
\begin{proof}[Proof of Theorem~\ref{main_theorem}]
	\noindent
	Let us first suppose that $\text{det} \ \mathcal{C}^{(k)}_{S_1,...,S_k} = 0$ and let $T$ be a $k$-trek system between $S_1,...,S_k$.
	\begin{itemize}
		\item If all elements of the multiset $\text{top}(T)$ are distinct, then Lemma~\ref{lemma_det_i_equal_0} implies that there exists an integer $i \in [k]$ such that $\text{det}(I -\Lambda)^{-1}_{top({T}),S_i} = 0$. By Lemma~\ref{GVL_lemma}, any path system from $\text{top}(T)$ to $S_i$ has a sided intersection, therefore ${T}$ has a sided intersection.
		\item If $\text{top}(T)$ has repeated elements, then at least two $k$-treks in $T$ intersect, namely at their top, hence ${T}$ has a sided intersection.
	\end{itemize}

	Conversely, let's suppose that every $k$-trek system $(T)$ between $S_1,...,S_k$ has a sided intersection. 
	Let's consider any set $R \subseteq V$ that satisfies $\#R = \#S_1 = ... = \#S_k$. 
	\begin{itemize}
		\item If $R$ forms the top of a $k$-trek system that ends at $S_1,...,S_k$, then there exists at least one integer $ i \in [k]$ such that there is a sided intersection in any path system between $R$ and $S_i$. By the Gessel-Viennot-Linstrom Lemma \ref{GVL_lemma}, $\text{det}(I-\Lambda)^{-1}_{R,S_i} = 0$, i.e.,  $\text{det}(I-\Lambda)^{-1}_{R,S_1} ...\ \text{det}(I-\Lambda)^{-1}_{R,S_k} = 0$ and therefore $\text{det } \mathcal E^{(k)}_R\text{det}(I-\Lambda)^{-1}_{R,S_1} ...\ \text{det}(I-\Lambda)^{-1}_{R,S_k} = 0$.
		\item Alternatively, if $R$ does not form the top of a $k$-trek system that ends at $S_1,...,S_k$, then there is no $k$-path system from $R$ to $S_1,...,S_k$, i.e., there is no path system between $R$ and at least one of $S_1,...,S_k$. This implies that at least one of $\text{det }(I-\Lambda)^{-1}_{R,S_i} = 0$ and therefore $\text{det } \mathcal E^{(k)}_R\text{det}(I-\Lambda)^{-1}_{R,S_1} ...\ \text{det}(I-\Lambda)^{-1}_{R,S_k} = 0$.
	\end{itemize}
	Thus, $\text{det }\mathcal{C}^{(k)}_{S_1,...,S_k} = \sum\limits_{R\subseteq V} \text{det } \mathcal{E}^{(k)}_{R} \text{det}(I - \Lambda)^{-1}_{R, S_1} \cdots \text{det}(I-\Lambda)^{-1}_{R, S_k} = 0$.
\end{proof}

\begin{example}\label{example_main_theorem} Theorem \ref{main_theorem} enables us to determine whether random variables have a common cause. Consider the graphs in Figures~\ref{fig:fig3} and \ref{fig:fig2} below.
\begin{figure}[H]
	\centering
	\begin{subfigure}[b]{0.2\linewidth}
		\includegraphics[scale=0.29]{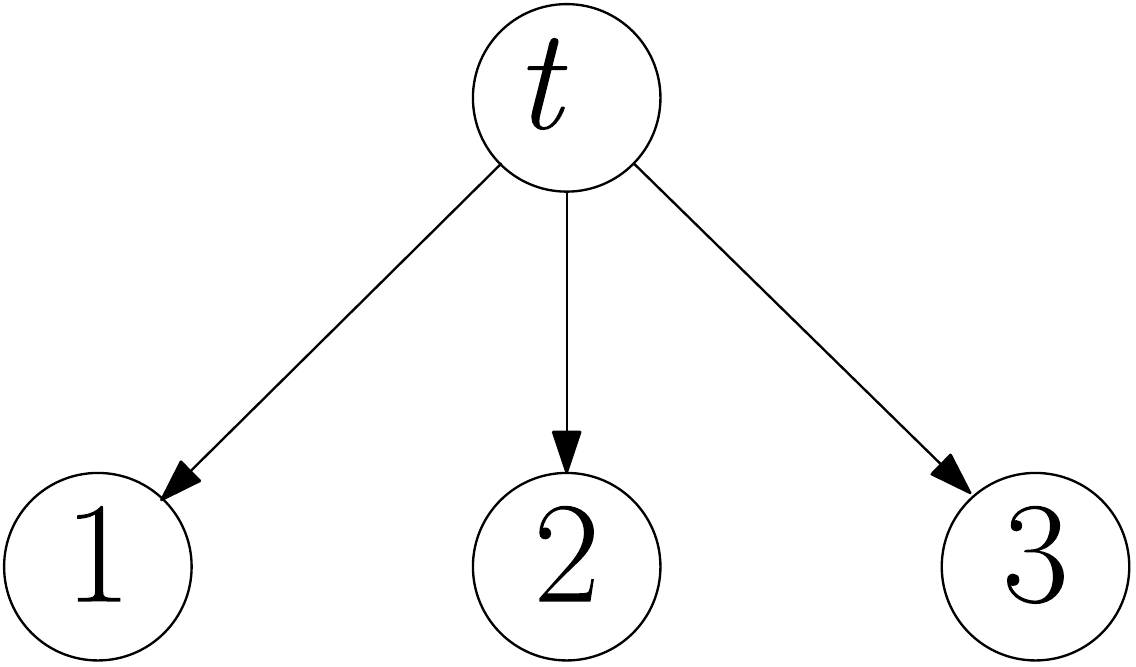}
	\caption{\label{fig:fig3}}
	\end{subfigure}\hspace{0.1\textwidth}%
	\begin{subfigure}[b]{0.2\linewidth}
		\includegraphics[scale = 0.5]{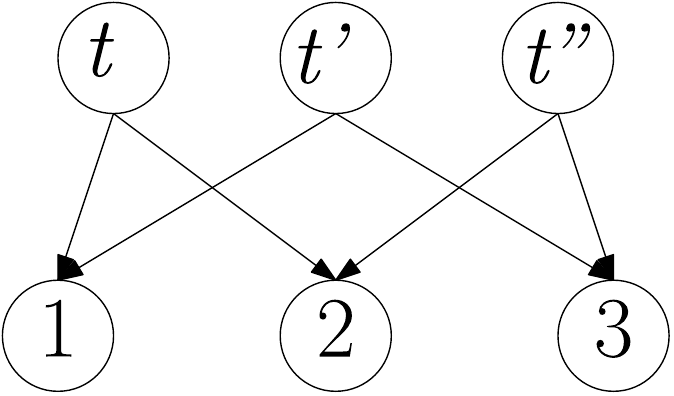}
	\caption{\label{fig:fig2}}
	\end{subfigure}
	\caption{}
\end{figure}
\noindent
Let $A = \{1\}$, $B = \{2\}$, and $C = \{3\}$. In Figure~\ref{fig:fig3} there is one $3$-trek joining $A,B$ and $C$, thus, det$(\mathcal{C}^{(3)}_{ABC}) = \mathcal{C}^{(3)}_{123}\neq$ 0. In Figure~\ref{fig:fig2} there is no $3$-trek joining $A,B$, and $C$ and, a fortiori, no $3$-trek without sided intersection. Therefore, det$(\mathcal{C}^{(3)}_{ABC}) = 0$.
\end{example}

The seminal paper~\cite{Sullivant2008} shows that the vanishing of determinants of the covariance matrix of $X$ is equivalent to a 2-trek separation criterion in the graph $G$. In the rest of this section, we illustrate that a generalization of this criterion to the case $k >2$ only works in one direction.
\noindent
\begin{definition}
\noindent
The collection of sets $(A_1,...,A_k)$ {\em $k$-trek-separates} $S_1,...,S_k$ if for every $k$-trek with paths $(P_{1}, ..., P_{k})$ between $S_1,...,S_k$, there exists $j \in \{1,...,k\}$ such that $P_j$ contains a vertex from $A_j$.
\end{definition} 

\begin{figure}[H]
	\centering
	\includegraphics[width=0.27\linewidth]{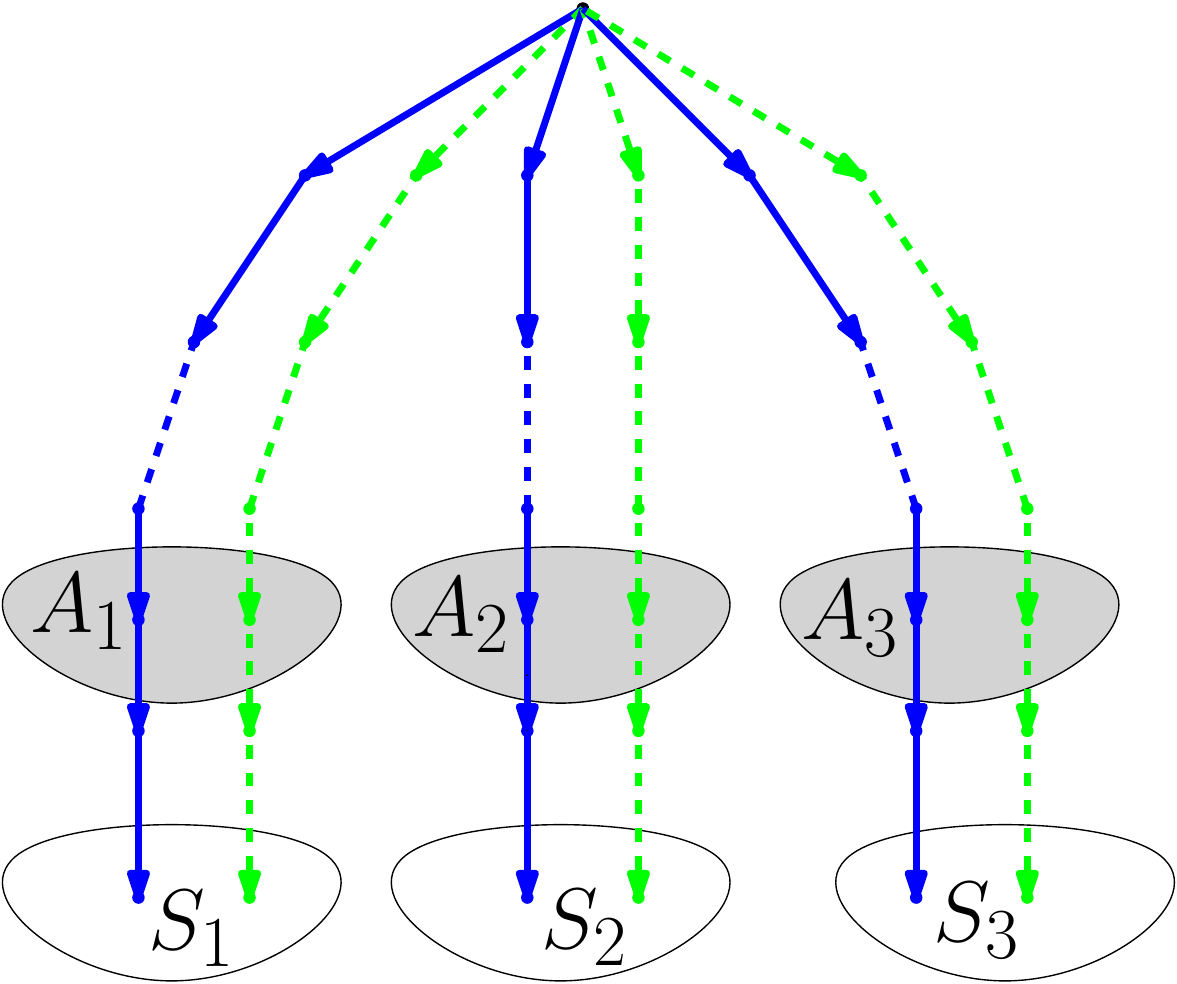}
	\caption{}
	\label{tri_trek_separation}
\end{figure}
\begin{example}
In Figure \ref{tri_trek_separation}, the sets $S_1, S_2$ and $S_3$ are 3-trek-separated by the sets $A_1, A_2$ and $A_3$.

\end{example}

\begin{theorem}[{\cite[Theorem~2.8]{Sullivant2008}}]
\noindent
The submatrix $\Sigma_{A,B}$ has rank less than or equal to r for all covariance matrices consistent with the graph $G$ if and only if there exist subsets $C_A, C_B \subset V$ with $\#C_A + \#C_B\leq r$ such that $(C_A,C_B)$ 2-trek-separates $A$ from $B$. Consequently, 
$$\text{rk}(\Sigma_{A,B}) \leq \text{min}\{\#C_A + \#C_B: (C_A,C_B) \text{ 2-trek-separates } A \text{ from } B\}$$ and equality holds for generic covariance matrices in the model $\mathcal M^{(2)}(G)$.
\end{theorem}

In Corollary \ref{corollary_k_trek_separation}, we show that when $k\geq 3$, $k$-trek-separation implies the vanishing of the corresponding cumulant tensor determinant (but not necessarily vice-versa).
\begin{corollary}\label{corollary_k_trek_separation}
\noindent
Consider $k$ sets of vertices $S_1,...,S_k$ with $ \#S_1 = ... = \#S_k = n$. For all tensors $C^{(k)}$ of $k^{th}$-order cumulants consistent with the graph G, the subtensor $\mathcal{C}^{(k)}_{S_1,...,S_k}$ has a null determinant if there exist subsets $A_1,..., A_k \subset V$ with $ \#A_1 + ... + \#A_k < n$ such that $(A_1,...,A_k)$ $k$-trek separates $S_1,..., S_k$.
\end{corollary}
\begin{proof}

Let us suppose that there exist $A_1,...,A_k$ such that $\#A_1 + ... + \#A_k < n$ and $(A_1,...,A_k)$ $k$-trek-separates $S_1,...,S_k$. Consider a $k$-trek system $T = (T_1,\ldots, T_n)$ between $S_1,\ldots, S_k$. Then, for every $i=1,\ldots, n$, there exists $m_i$ such that the $m_i$-th component of $T_i$ intersects $A_{m_i}$. Since $\#A_1 + \cdots+\#A_k <n
$, by the Pigeon-Hole Principle, there exist $i\neq j$ such that $m=m_i=m_j$, and the $m$-th components of $T_i$ and $T_j$ go through the same element $s\in A_m$. Therefore, $T$ has a sided intersection. Thus, every $k$-trek system between $S_1,\ldots, S_k$ has a sided intersection, and, therefore, by Theorem~\ref{main_theorem}, det $\mathcal{C}^{(k)}_{S_1,...,S_k} = 0$. 

\end{proof}
\noindent
\begin{remark}
When $k \geq 3$, the reverse implication of Corollary~\ref{corollary_k_trek_separation} is not true, i.e., \text{det }($\mathcal{C}^{(k)}_{S_1,...,S_k}) = 0$ does not imply that there exists $(A_1,...,A_k)$ such that $\#A_1 + ... + \#A_k < n$ and $(A_1,...,A_k)$ $k$-trek-separates $(S_1,...,S_k)$. Consider the graph in Figure~\ref{counterexample_t_separation} below. In this graph, $\#S_1 = \#S_2 = \#S_3 = 2$. There is no system of two 3-treks between $S_1, S_2, S_3$ without sided intersection. However, it is not possible to find three sets $A_1, A_2, A_3$ such that $\#A_1 + \#A_2 + \#A_3 < 2 = \#S_i$, i.e., there is no such set that 3-trek-separates $S_1,S_2,S_3$. Intuitively, this happens because "max flow" and "min cut" are no longer equal, while the reason the statement holds when $k=2$ is precisely because of the min-cut max flow (Menger's) Theorem \cite{Menger1927}, which was used in the proof of the trek-separation theorem by \cite{Sullivant2008}.

\begin{figure}[H]
\centering
\includegraphics[width=0.35\linewidth]{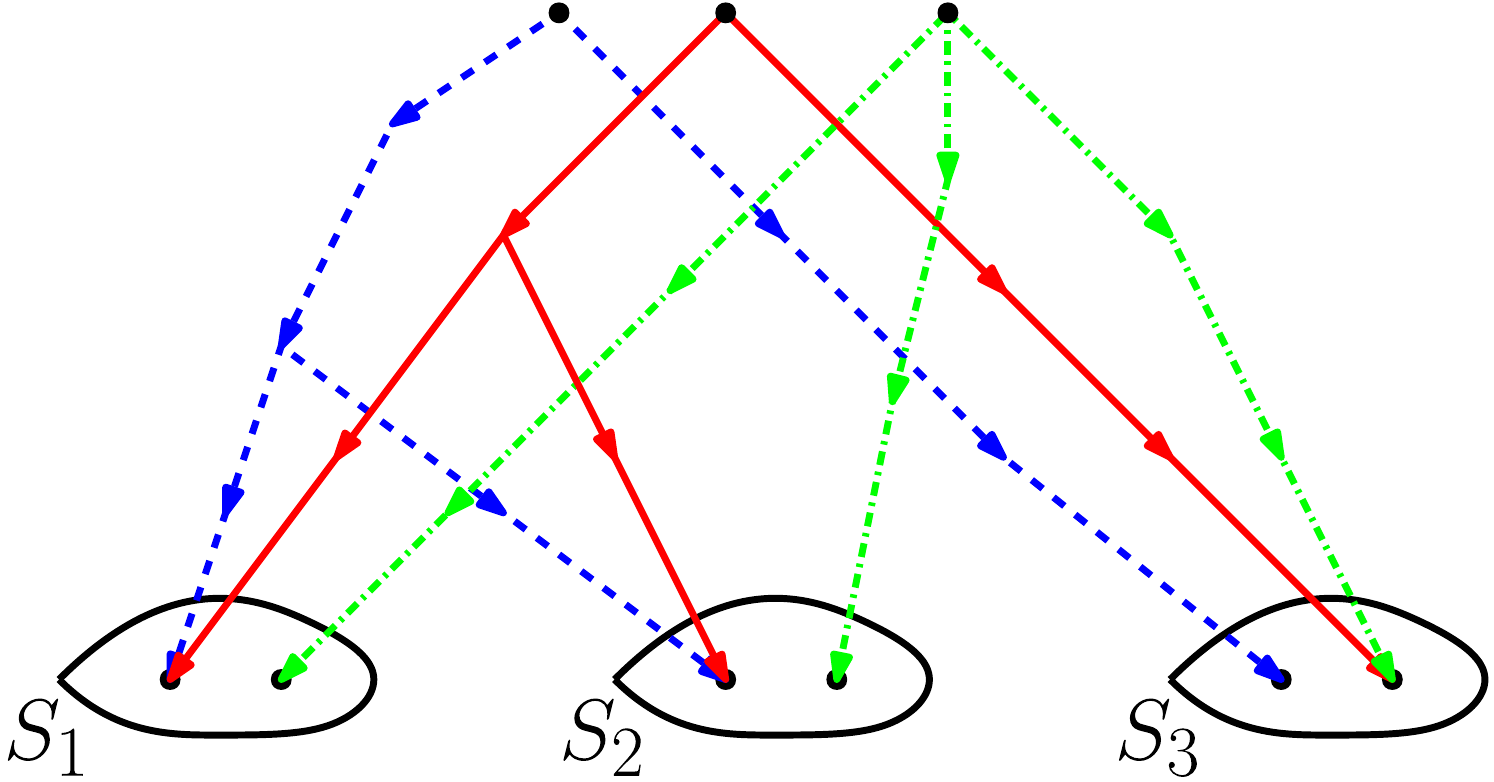}
\caption{}
\label{counterexample_t_separation}
\end{figure}
\end{remark}
\section{Hidden variables}\label{hidden variables}\label{section_4}
We now consider the case where the structural equation model (\ref{general equation_LSEM}) involves some hidden (i.e., unobserved) variables. Alternatively, it is similar to think of  the noise variables $\varepsilon_i$ as correlated. We represent such a case with a mixed graph $G = (V,\mathcal{D},\mathcal{H})$, where $V$ is the set of vertices, $\mathcal{D} \subseteq V\times 
V$  is the set of directed edges, and $\mathcal{H}$ is the set of \emph{multidirected edges} (see Definition~\ref{multidirected_edge} below), which signify the dependencies between the $\varepsilon$ variables. We assume that $G$ does not contain any cycle nor loop.

\begin{definition}\label{multidirected_edge}
A multidirected edge between nodes $i_1,...,i_k$ is the union of $k$ directed edges with the same source and with sinks $i_1,...,i_k$. The $k$-directed edges are merged at their source without an additional node. We call $k$ the \emph{order} of the multidirected edge.
\end{definition}
\noindent

\begin{figure}[h!] 
\centering
\includegraphics[width=0.2\linewidth]{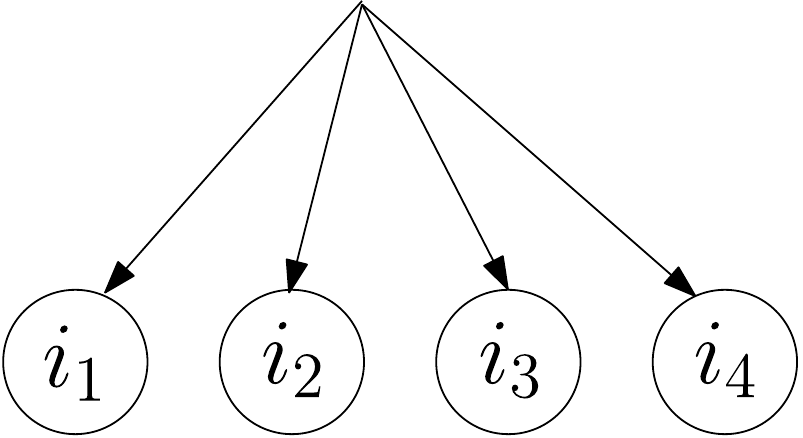} 
\caption{A multidirected edge between $i_1,i_2,i_3,i_4$ of order 4.}
\end{figure}

\begin{remark}
Note that this is a generalization of the notion of bidirected edges widely used in the literature. \end{remark}

\noindent
As before, define $\mathcal{E}^{(k)}$ to be the tensor of $k$-th order cumulants of the variables $\varepsilon_i$. Since we do not assume the independence of the variables $\varepsilon_i$, then any entry of the tensor $\mathcal{E}^{(k)}$ may be non-zero. Specifically, $\mathcal{E}^{(k)}$ is a $\underbrace{V \times V \times ... \times V}_k$ tensor for which the entry $w_{i_1,...,i_k}$ is non-zero if there exists a multidirected edge between $a_1,...,a_l$ such that $\{i_1,...,i_k\} \subseteq \{a_1,...,a_l\}$. Note that some of the elements $i_1,\ldots, i_k$ could be equal.

We now define the notion of a $k$-trek in this setting.
\begin{definition}
A $k$-trek between vertices $i_1,...,i_k$ in a mixed graph $G$ is composed of $k$ directed paths $(P_1,\ldots, P_k)$ where $P_s$ goes from $j_s$ to $i_s$, and
the {\em tops} $j_1,\ldots, j_k$ are connected in one of the following ways.
\begin{enumerate}
    \item[(a)] Either the vertices $j_1,\ldots, j_k$ coincide to form the top of the trek; or
    \item[(b)] There exist an $l$-directed edge between $a_1,...,a_l$ such that $\{j_1,...,j_k\} \subseteq \{a_1,...,a_l\}$.
\end{enumerate}
\end{definition}

\begin{example}
In the mixed graph from Figure~\ref{example_k_trek_hidden_variables_case_1},
$(1 \rightarrow 7, 1 \rightarrow 6, 1 \rightarrow 4 \rightarrow 5)$ is a 3-trek between 7, 6 and 5.
The tops of the paths going to the nodes 7, 6 and 5 coincide with node 1.    
In the mixed graph from Figure~\ref{example_k_trek_hidden_variables_case_2}, 
$(7, 6, 4 \rightarrow 5)$ is a 3-trek between 7, 6 and 5. The tops of the paths are connected by a 3-directed edge between $4,6,$ and $7$.
\begin{figure}[H]
\centering
	\begin{subfigure}[b]{0.4\linewidth}
     \includegraphics[scale=0.47]{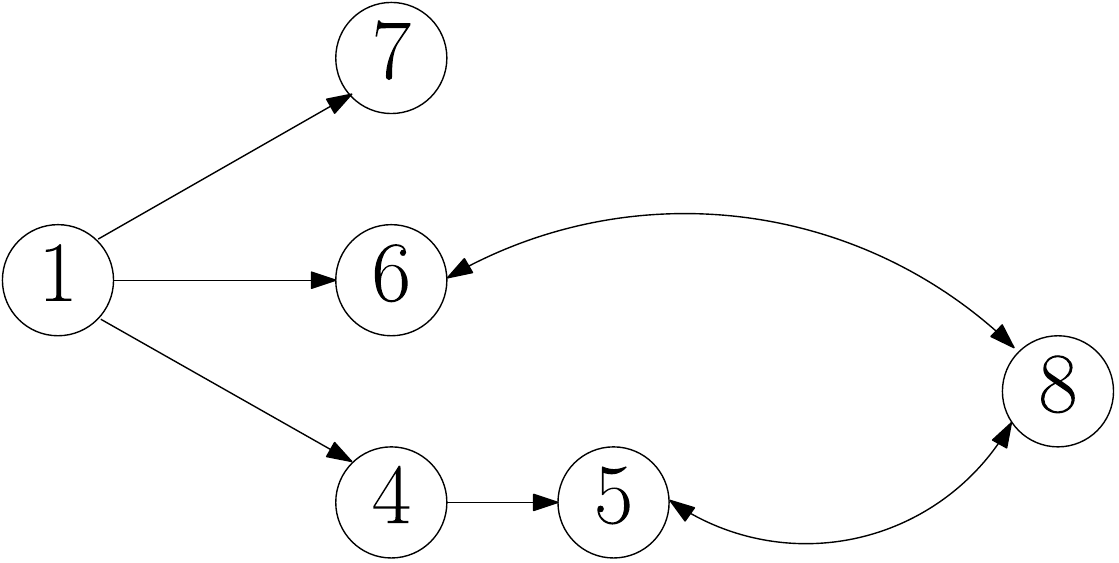}
     \caption{ \label{example_k_trek_hidden_variables_case_1}}
    \end{subfigure}
	\begin{subfigure}[b]{0.4\linewidth}
	\includegraphics[scale=0.2]{figure_DAG_intro_hidden_multi_edge.pdf}
	\caption{\label{example_k_trek_hidden_variables_case_2}}
\end{subfigure}
\caption{}
\end{figure}
\end{example}

\begin{definition}
The linear structural equation model $X_j = \sum_{i\in\text{pa}(j)} \lambda_{ij}X_i + \varepsilon_j$ given by a mixed graph $G = (V,\mathcal D,\mathcal H)$, where $V=[p]$, is the family of distributions on $\mathbb{R}^p$ with tensor of $k^{th}$-order cumulants in the set:
\begin{equation}
    \mathcal{M}^{(k)}(G) = \{\mathcal{E}^{(k)}\bullet (I-\Lambda)^{-k} : \Lambda = (\lambda_{ij}) \in \mathbb{R}^{\mathcal D}, \mathcal E^{(k)}\in(\mathbb R^{\otimes k})^\mathcal H\}
\end{equation}
where $(\mathbb R^{\otimes k})^\mathcal H$ denotes the set of $p\times\cdots\times p$ ($k$-times) cumulant tensors of $\varepsilon$, which are zero at all entries $(i_1, \ldots, i_k)$ unless $i_1=\cdots =i_k$ or  there exists a multi-directed edge $(j_1,\ldots, j_\ell)\in\mathcal H$ such that $\{i_1,\ldots, i_k\}\subseteq \{j_1,\ldots, j_\ell\}$.
\end{definition}

 Extending the $k$-trek rule, Proposition \ref{prop_equation_entry_C}, to the hidden variable case, we show that every entry of the tensor $\mathcal{C}^{(k)}$ can be expressed as a sum of $k$-trek monomials as well. 
\begin{corollary}
For a noise vector $\varepsilon$ whose entries are dependent, the entries of the $k$-th order cumulant tensor $\mathcal C^{(k)}$ of $X$ can be expressed as a sum over {\em $k$-trek monomials},
\begin{equation}\label{eq:equation_entry_C_hidden_case}
\mathcal{C}^{(k)}_{i_1,...,i_k} = \sum_{(P_1,...,P_k) \in \mathcal{T}(i_1,...,i_k)} {\mathcal E^{(k)}_{top(P_{1}),..., top(P_{k})} \lambda^{P_{1}}} ... \ \lambda^{P_{k}},
\end{equation}
where $\mathcal{T}{(i_1,...,i_k)}$ is the set of all $k$-treks between $i_1,...,i_k$.
\end{corollary}

\begin{proof}

By Lemma~\ref{factorization_cumulant_tensor}, we can express the $k$-th cumulant tensor of $X$ via the Tucker decomposition $\mathcal{C}^{(k)} = 
\mathcal{E}^{(k)}\bullet(I - \Lambda)^{-k}$. Therefore,
\begin{align*}
\mathcal C^{(k)}_{i_1,\ldots, i_k} &= \sum_{j_1,\ldots, j_k}\mathcal E^{(k)}_{j_1,\ldots, j_k} ((I-\Lambda)^{(-1)})_{j_1,i_1}\cdots ((I-\Lambda)^{(-1)})_{j_k,i_k}\\
&= \sum_{j_1,\ldots, j_k}\mathcal E^{(k)}_{j_1,\ldots, j_k} (\sum_{P_1\in\mathcal P(j_1, i_1)}\lambda^{P_1})\cdots (\sum_{P_k\in\mathcal P(j_k, i_k)}\lambda^{P_1})\\
&= \sum_{{j_1,\ldots, j_k,} \atop {P_1\in\mathcal P(j_1,i_1), \ldots, P_k\in\mathcal P(j_k, i_k)}}\mathcal E^{(k)}_{j_1,\ldots, j_k}\lambda^{P_1}\cdots \lambda^{P_k}.
\end{align*}
Note that $\mathcal E^{(k)}_{j_1,\ldots, j_k}$ is the cumulant of $X_{j_1},\ldots, X_{j_k}$ which is nonzero if and only if there exists a multidirected edge connecting (at least) $j_1,\ldots, j_k$, which is exactly when $j_1,\ldots, j_k$ could be at the top of a $k$-trek. Therefore,
\begin{align*}
   \mathcal C^{(k)}_{i_1,\ldots, i_k} =  \sum_{(P_1,...,P_k) \in \mathcal{T}(i_1,...,i_k)} {\mathcal E^{(k)}_{top(P_{1}),..., top(P_{k})} \lambda^{P_{1}}} ... \ \lambda^{P_{k}}.
\end{align*}
\end{proof}

\noindent
In a similar fashion to the directed acyclic case, we obtain the following result.
\begin{theorem}\label{main_theorem_hidden_variables}
\noindent
Consider a mixed graph $G(V,\mathcal{D},\mathcal{H})$. Let $S_1,..., S_k$ be subsets of $V$ with $\#S_1 = ... = \#S_k$. Then,
$$ \text{det } \mathcal{C}^{(k)}_{S_1,..., S_k} = 0$$ if and only if every system of $k$-treks between $S_1,..., S_k$ has a sided intersection.
\end{theorem}

\begin{proof}
Theorem \ref{main_theorem_hidden_variables} extends Theorem \ref{main_theorem} to the case of mixed graphs. We use a common argument in the graphical models literature that enables us to convert the multidirected case to the directed case. We replace every multidirected edge joining $i_1,...,i_k$ with a vertex $v$ and a directed edge between $v$ and each
of the vertices $i_1,...,i_k$. We call this graph $\widetilde{G}$, also known as the {\em canonical DAG} associated to $G$.

\begin{figure}[H]
	\centering
	\begin{subfigure}[b]{0.2\linewidth}
		\includegraphics[width=\linewidth]{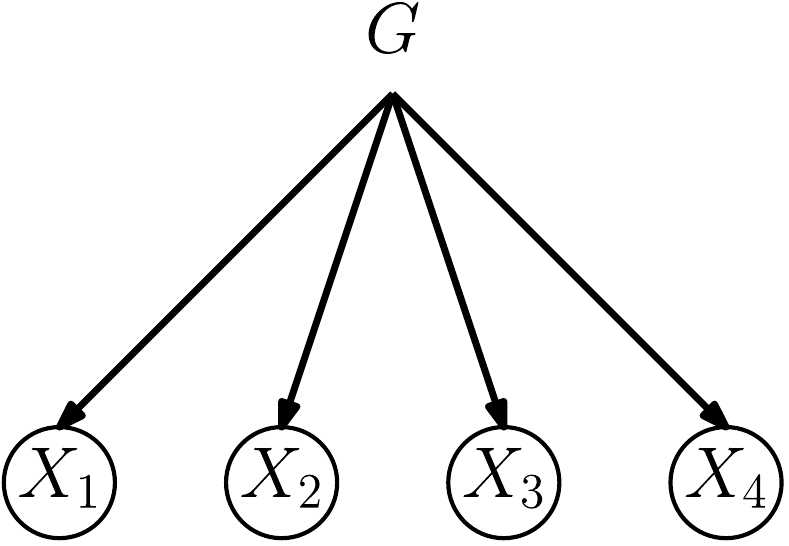}
	\caption{\label{fig1_canonical_DAG}}
	\end{subfigure}\hspace{0.1\textwidth}%
	\begin{subfigure}[b]{0.2\linewidth}
		\includegraphics[width=\linewidth]{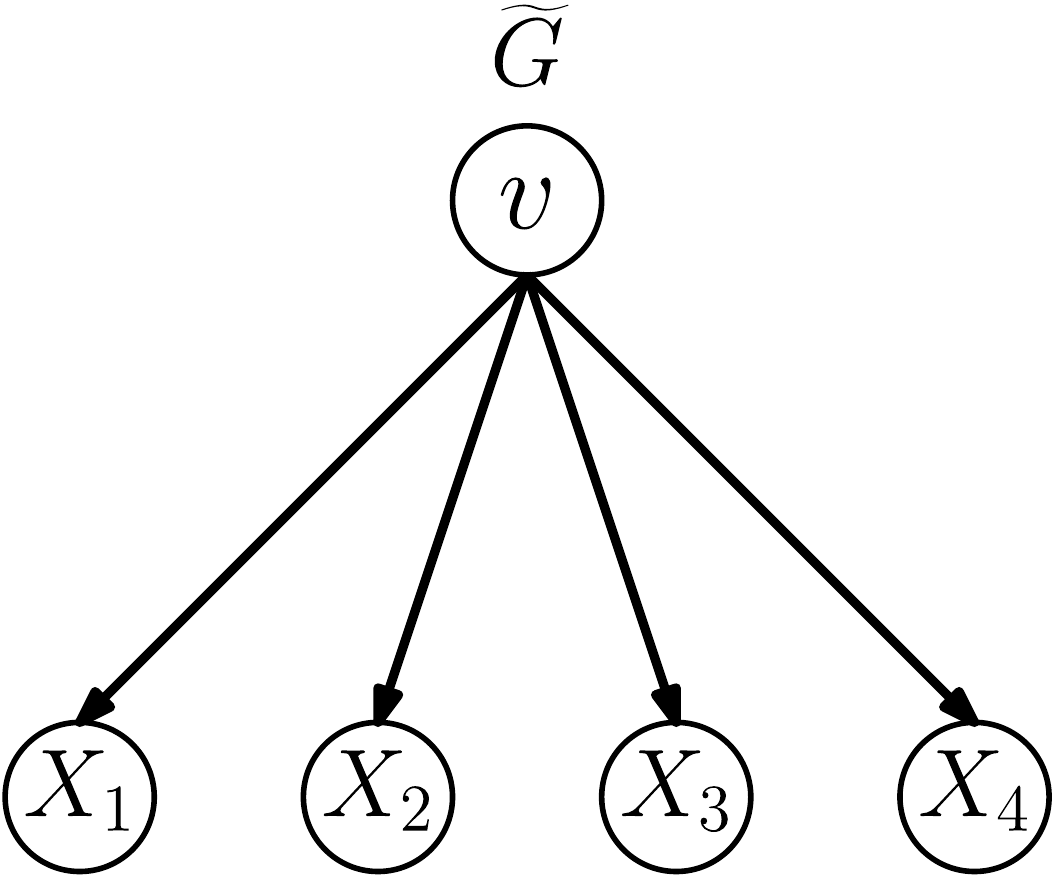}
	\caption{\label{fig2_canonical_DAG}}
	\end{subfigure}\hspace{0.1\textwidth}%
	\caption{\textit{Mixed graph $G$ in (a) and its corresponding canonical DAG $\widetilde{G}$ in (b)}}
\end{figure}

\begin{proposition}\label{proposition_reduce_directed_case}
Let $S_1,...,S_k \subset V$ be $k$ sets of vertices such that $\#S_1=...=\#S_k$. Then the determinant of  $\mathcal{C}^{(k)}_{S_1,...,S_k}$ is zero for all cumulant tensors $\mathcal{C} \in \mathcal{M}^{(k)}(G)$ if and only if the determinant of  $\widetilde{\mathcal{C}}^{(k)}_{S_1,...,S_k}$ is zero for all cumulant tensors $\widetilde{\mathcal{C}} \in \mathcal{M}^{(k)}(\widetilde{G})$. 
In other words, if there is no $k$-trek system without sided intersection between $S_1,...,S_k$ in $G$, then there is no $k$-trek system without sided intersection between $S_1,...,S_k$ in $\widetilde{G}$, and vice-versa.
\end{proposition}

\noindent
\begin{proof}
Let $\mathcal{M}^{(k)}(G)$ and $\mathcal{M}^{(k)}(\widetilde{G})$ be the sets of $k^{th}$-order cumulant tensors of $G$ and  $\widetilde G$, respectively. We will prove that $\mathcal{M}^{(k)}(G)$ and $\mathcal{M}^{(k)}(\widetilde{G})$ have the same Zariski closure, i.e., a polynomial equation vanishes on $\mathcal{M}^{(k)}(G)$ if and only if it vanishes on $\mathcal{M}^{(k)}(\widetilde{G})$.  Note that $\tilde{G}$ has more vertices than $G$, so when comparing the Zariski closures of $\mathcal{M}^{(k)}(G)$ and $\mathcal{M}^{(k)}(\tilde G)$, we implicitly assume that we are only considering cumulants between vertices in $G$, i.e., we are projecting the cumulants onto the vertices of $G$.  To do this, let us show that the two parametrizations give the same family of tensors near the identity tensor. 
Note that there exist distributions of independent variables $X_1, ..., X_p$ for which the $k^{th}$-order cumulant is the identity tensor. As shown in Proposition~\ref{equation_entry_T}, $\mathcal{C}^{(k)}_{i_1,...i_k}$ can be expressed as a sum of the trek monomials of all $k$-treks in $\mathcal{T}(i_1,...,i_k)$ between $i_1,...,i_k$ in $G$ as $\mathcal{C}^{(k)}_{i_1,...,i_k} = \sum\limits_{(P_1,...,P_k) \in \mathcal{T}(i_1,...,i_k)} {\mathcal{E}^{(k)}_{top(P_{1},..., P_{k})} \lambda^{P_{1}}} ... \ \lambda^{P_{k}}$. Similarly, $\widetilde{\mathcal{C}}^{(k)}_{i_1,...,i_k}$ is the sum of the trek monomials of all $k$-treks in $\widetilde{\mathcal{T}}(i_1,...,i_k)$ between $i_1,...,i_k$ in $\widetilde{G}$, i.e., $\widetilde{\mathcal{C}}^{(k)}_{i_1,...,i_k} = \sum\limits_{(\widetilde{P}_1,...,\widetilde{P}_k) \in \widetilde{\mathcal{T}}(i_1,...,i_k)} {\mathcal{E}^{(k)}_{top(\widetilde{P}_{1},..., \widetilde{P}_{k})} \lambda^{\widetilde{P}_{1}}} ... \ \lambda^{\widetilde{P}_{k}}$. 

Now, let's set
\begin{align}
    \mathcal{E}^{(k)}_{i_1,...,i_k} &= \widetilde{\mathcal{E}}^{(k)}_{v,...,v} \  \widetilde{\lambda}_{v,i_1} ... \widetilde{\lambda}_{v,
    i_k}\label{equation_epsilon_1} \intertext{if there is a multi-directed edge between $i_1,\ldots, i_k$ in $G$, i.e., there is a directed edge from the vertex $v$ to each of the nodes $i_l, l \in \{1, \ldots, k\}$ in $\widetilde G$, and let
    }        \mathcal{E}^{(k)}_{i,...,i} &= \widetilde{\mathcal{E}}^{(k)}_{i,...,i} + \sum_{v \in \mathcal H } \widetilde{\mathcal{E}}^{(k)}_{v,...,v} \widetilde{\lambda}^k_{v,i}\label{equation_epsilon_2}
\end{align}
\noindent
for each $i\in [p]$.
As we initially assumed we are near the identity tensor $\mathcal{I}^{(k)}$, we can switch from one parametrization to the other as follows. First, note that $\mathcal{I}^{(k)} \in \mathcal{M}^{(k)}(G)$ and $\mathcal{I}^{(k)} \in \mathcal{M}^{(k)}(\widetilde{G})
$. Then, given $\widetilde{\mathcal{E}}^{(k)}$ and $\widetilde{\lambda}$, we can find $\mathcal{E}^{(k)}$ from equations (\ref{equation_epsilon_1}) and (\ref{equation_epsilon_2}), and $\lambda_{ij} = \widetilde{\lambda}_{ij}$ for $ i \rightarrow j$, i.e., given $\widetilde{\mathcal{C}}^{(k)} \in \mathcal{M}^{(k)}(\widetilde{G})$, using equation (\ref{equation_epsilon_1}) and equation (\ref{equation_epsilon_2}), we get that the corresponding $\mathcal{C}^{(k)} \in \mathcal{M}^{(k)}(G)$, therefore $\mathcal{M}^{(k)}(\widetilde{G}) \subseteq \mathcal{M}^{(k)}(G)$. Conversely, given $\mathcal{E}^{(k)}_{i_1,...,i_k}$ and $\lambda_{ij}$ small enough, we can choose $\widetilde{\mathcal{E}}^{(k)}_{v,...,v} = \varepsilon > 0$, $ \widetilde{\lambda}_{ij}=\lambda_{ij}$ for $i \rightarrow j$ in $\mathcal D$, and $\widetilde{\lambda}_{v,i_l} = \sqrt[k]{\frac{|\mathcal{E}^{(k)}_{i_1,...,i_k}|}{\varepsilon}}$ for $l \in \{1,\ldots,k\}$. Therefore, equation (\ref{equation_epsilon_1}) is satisfied. Since $\mathcal{E}^{(k)}_{i_1,...,i_k}$ is small and $\mathcal{E}^{(k)}_{i,...,i}$ is near 1, we can find  $\widetilde{\mathcal{E}}^{(k)}_{i,...,i} > 0$ such that equation (\ref{equation_epsilon_2}) is satisfied and such that $\widetilde{\mathcal{E}}^{(k)}$ is a diagonal tensor. This shows that if $\mathcal{C}^{(k)} \in \mathcal{M}^{(k)}(G)$ is in a neighborhood of $\mathcal{I}^{(k)}$, then we can find the corresponding $\widetilde{\mathcal{C}}^{(k)} \in \mathcal{M}^{(k)}(\widetilde{G})$ as well. Note that this correspondence between $\mathcal{C}^{(k)}$ and $\widetilde{\mathcal{C}}^{(k)}$ is a bijection. Thus $\mathcal{M}^{(k)}(G)$ and $\mathcal{M}^{(k)}(\widetilde{G})$ are equal in an open neighborhood and so they have the same Zariski closure, i.e., $\overline{\mathcal{M}^{(k)}(G)} = \overline{\mathcal{M}^{(k)}(\widetilde{G})}$, where $\overline{\mathcal{M}^{(k)}(G)}$ denotes the Zariski closure of $\mathcal{M}^{(k)}(G)$. Therefore the determinant of $\widetilde{\mathcal{C}}^{(k)}_{S_1,...,S_k}$ vanishes on $\mathcal{M}^{(k)}(\widetilde G)$  if and only if the determinant of $\mathcal{C}^{(k)}_{S_1,...,S_k}$ vanishes on $\mathcal{M}^{(k)}(G)$. This is equivalent to saying that there is no system of $k$-treks without sided intersection between $S_1,...,S_k$ in $\widetilde{G}$ if and only if there is no system of $k$-treks without sided intersection between $S_1, \ldots, S_k$ in $G$. 
\end{proof}

Going back to the proof of Theorem~\ref{main_theorem_hidden_variables}, Proposition~\ref{proposition_reduce_directed_case} enables us to reduce the multidirected setting to a directed acyclic graph for which we proved Theorem~\ref{main_theorem}. 
\end{proof}

\noindent
\begin{example}\label{example_main_theorem_hidden_variables_statement} The following example shows that Theorem \ref{main_theorem_hidden_variables} enables us to determine whether random variables have a common cause.
\begin{figure}[H]
	\centering
	\begin{subfigure}[b]{0.2\linewidth}
		\includegraphics[scale=0.5]{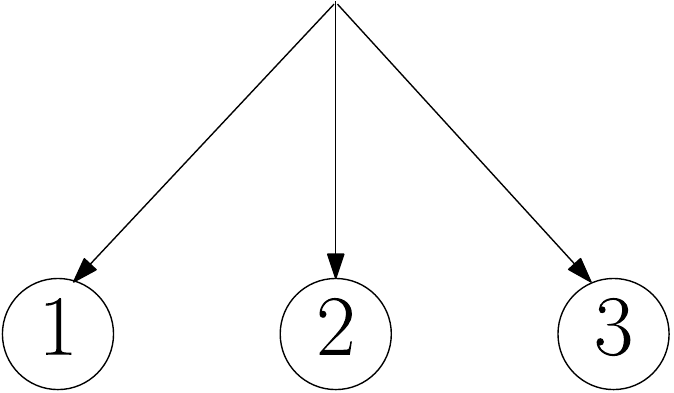}
	\caption{\label{fig:fig1a}}
	\end{subfigure}\hspace{0.1\textwidth}%
	\begin{subfigure}[b]{0.2\linewidth}
		\includegraphics[scale=0.48]{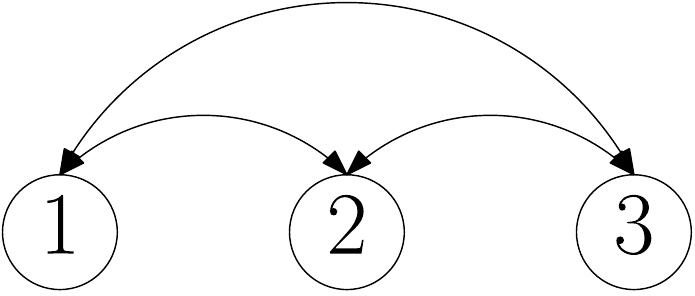}
	\caption{\label{fig:fig1b}}
	\end{subfigure}
	\caption{}
	\label{example_main_theorem_hidden_variables_case}
\end{figure}
\noindent In the two graphs from Figure~\ref{example_main_theorem_hidden_variables_case}, let $A = \{1\}$, $B = \{2\}$, and $C = \{3\}$. In Figure~\ref{fig:fig1a}, there is one $3$-trek joining $A,B$, and $C$, hence det$(\mathcal{C}^{(3)}_{ABC}) = \mathcal{C}^{(3)}_{123}\neq$ 0. In Figure~\ref{fig:fig1b}, there is no $3$-trek joining $A,B$ and $C$ and, a fortiori, no $3$-trek without sided intersection. Therefore det$(\mathcal{C}^{(3)}_{ABC}) = 0$.
\end{example}

\section{Determinants of higher-order moments and multi-trek systems}\label{section_5}

At the start of this project, we focused on tensors of higher-order moments rather than cumulants. However, contrary to cumulant tensors (cf. Lemma \ref{lem:diagonal_cumulants}), moment tensors of order greater than 3 are not diagonal when the variables are independent. Extending Theorem~\ref{main_theorem} to higher-order moment tensors is therefore not straightforward. Nevertheless, we conjecture that the result still holds. Before stating this conjecture, we translate our results obtained for cumulant tensors to moment tensors. 

Let $\Phi^{(k)}$ be the tensor of $k^{th}$-order moments of $\varepsilon$. Then, the $k^{th}$-order moment tensor of $X$ is given as follows.

\begin{proposition}[{\cite[Chapter~5, Eq. (5.7)]{Comon2010}}]\label{expression_tensor}
\noindent
The tensor $\mathcal{N}^{(k)}$ of $k^{th}$-order moments of the random vector $X$ with mean $(0,...,0)$ equals
\begin{align}
\mathcal{N}^{(k)} &= \Phi^{(k)} \bullet (I-\Lambda)^{-k}\label{equation_T_Phi_Lambda}.
\end{align}
\end{proposition}

Let $G = (V, \mathcal{D},\mathcal{H})$ as in Section \ref{section_4}.
In order to account for the non-zero off-diagonal entries in the tensor $\Phi^{(k)}$ of higher-order moments of $\varepsilon$, we need to adapt our definition of a $k$-trek as follows.
\noindent
\begin{definition}\label{def_trek_moment}
\noindent
A $k$-split-trek in $G$ between $k$ nodes $v_1,\ldots, v_k$ is either:
\begin{itemize}
    \item[(a)] an ordered collection of $k$ directed paths $(P_{1}, ..., P_{k})$ where $P_{i}$ has sink $v_i$, and $P_{1},..., P_{k}$ have the same source; or 
    \item[(b)] an ordered collection of $k$ directed paths $(P_{1}, ..., P_{k})$ where $P_{i}$ has sink $v_i$, and $P_{1},..., P_{k}$ may have different sources, but each source must be shared by at least two paths.
\end{itemize}

\begin{figure}[H]
	\centering
	\begin{subfigure}[b]{0.2\linewidth}
		\includegraphics[width=\linewidth]{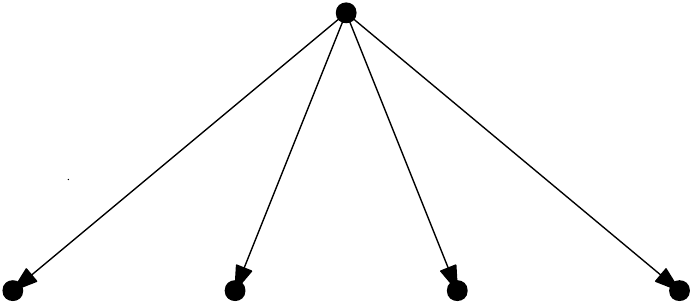}
	\caption{\label{fig:k-trek_moment_1}}
	\end{subfigure}\hspace{0.1\textwidth}%
	\begin{subfigure}[b]{0.2\linewidth}
		\includegraphics[width=\linewidth]{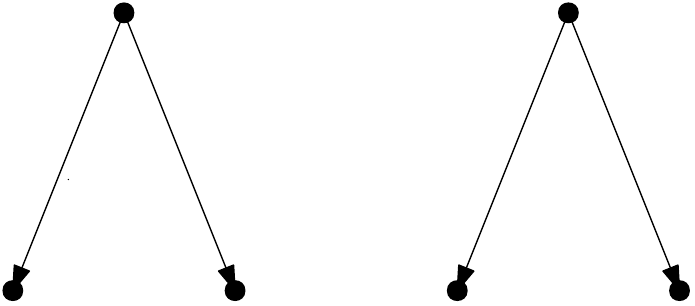}
	\caption{\label{fig:k-trek_moment_2}}
	\end{subfigure}\hspace{0.1\textwidth}%
	\begin{subfigure}[b]{0.2\linewidth}
	\includegraphics[width=\linewidth]{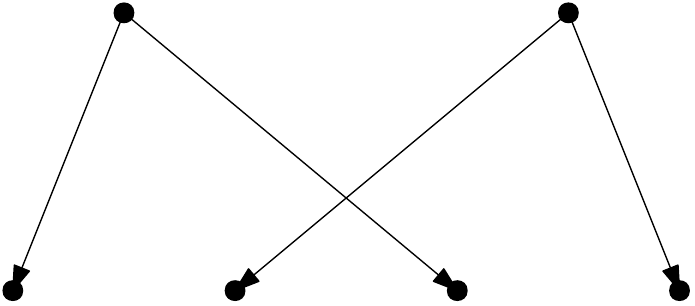}
	\caption{\label{fig:k-trek_moment_3}}
	\end{subfigure}\hspace{0.1\textwidth}%
	\begin{subfigure}[b]{0.2\linewidth}
	\includegraphics[width=\linewidth]{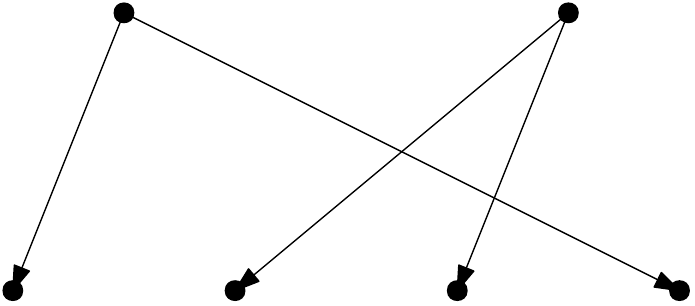}
	\caption{\label{fig:k-trek_moment_4}}
	\end{subfigure}\hspace{0.1\textwidth}%
	\begin{subfigure}[b]{0.2\linewidth}
	\includegraphics[width=\linewidth]{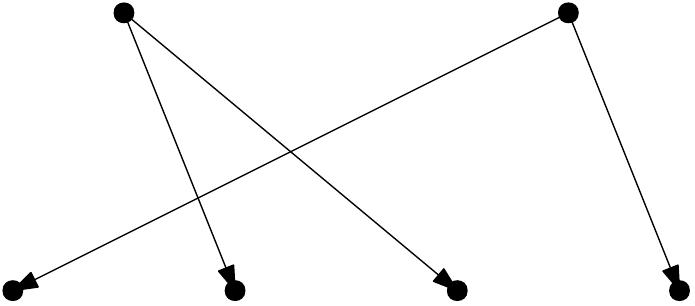}
	\caption{\label{fig:k-trek_moment_5}}
	\end{subfigure}\hspace{0.1\textwidth}%
	\begin{subfigure}[b]{0.2\linewidth}
	\includegraphics[width=\linewidth]{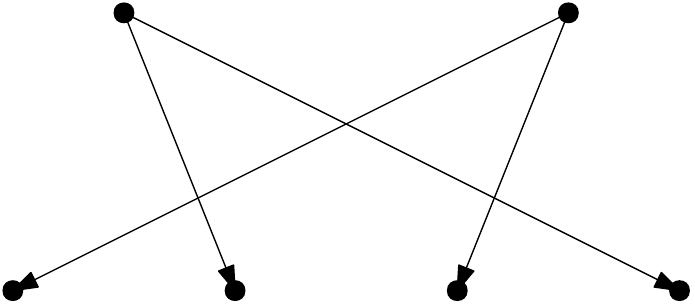}
	\caption{\label{fig:k-trek_moment_6}}
	\end{subfigure}\hspace{0.1\textwidth}%
	\caption{\textit{Types of $k$-split-treks for $k=4$}}\label{fig_types_of_treks}
\end{figure}

\end{definition}
\noindent

\begin{definition}\label{definition_centralized_distributed}
\noindent
The {\em top} of a $k$-split-trek $(P_1,\ldots, P_k)$ in $G$ is either:
\begin{itemize}
    \item[(a)] a node that is the source of each of the $k$ paths $P_1,\ldots, P_k$. 
    \item[(b)] a set of $j$ nodes where each of the $j$ nodes is a source of $\ell_j$ of the paths $P_1,\ldots, P_k$ such that $\sum_j \ell_j = k$ and $\ell_j\geq 2$ for all $j$.
\end{itemize}
\end{definition}
\noindent Note that when $k=3$, the notions of a 3-trek and a 3-split-trek coincide.
\begin{example}In Figure~\ref{fig_types_of_treks},  subfigure (a) illustrates the first type of a top, while subfigures (b-f) illustrate the second type.
\end{example}
\noindent
Using equation (\ref{entries_I-Lambda}), we rewrite the entries of  $(I-\Lambda)^{-1}$ in equation (\ref{equation_T_Phi_Lambda}) and thus express the entries of the tensor of $k^{th}$-order moments as follows.
\noindent
\begin{proposition}\label{equation_entry_T} We have that
\noindent
\begin{equation}
\mathcal{N}^{(k)}_{i_1,...,i_k} = \sum_{(P_1,...,P_k) \in \mathcal{S}(i_1,...,i_k)} {\Phi^{(k)}_{top(P_{1},..., P_{k}),\ldots, top(P_{1},..., P_{k})} \lambda^{P_{1}}} ... \ \lambda^{P_{k}} \label{expression_component_tensor}
\end{equation}
where $\mathcal{S}(i_1,...,i_k)$ is the set of all $k$-split-treks between $i_1,...,i_k$.
\end{proposition}

\begin{proof}
\noindent
From equation (\ref{equation_T_Phi_Lambda}), we have that $\mathcal{N}^{(k)} = \Phi^{(k)} \bullet (I-\Lambda)^{-k}$.
The entries of $\Phi^{(k)}$ are non-zero if and only if they are of the form $\mathbb{E}[\varepsilon_1^{x_1} ... \varepsilon_k^{x_k}]$ where $x_1, ..., x_k$ are integers either equal to 0, or strictly greater than 1. Such entries correspond to the cases when there is a $k$-split-trek between $i_1,..., i_k$ as defined in Definition \ref{def_trek_moment}. Furthermore, we have $(I-\Lambda)^{-1}_{ij} = \sum_{P \in \mathcal{P}(ij)} \lambda^P$ by equation (\ref{entries_I-Lambda}), and replacing this expression in equation (\ref{equation_T_Phi_Lambda}), we obtain equation (\ref{expression_component_tensor}), which completes the proof. \end{proof}

Similarly to the case of higher-order cumulants, the determinant of the subtensor of $k^{th}$-order moments indexed by the sets $S_1,...,S_k$ can be rewritten in terms of the split-trek systems between $S_1,...,S_k$.
\noindent
\begin{proposition}\label{det_T_trek_system_m_moments}
\noindent
Let $S_1,\ldots, S_k\subseteq V$ be $k$ sets of nodes such that $\#S_1 =...=\#S_k = n$. Then,
$$\text{det } \mathcal N^{(k)}_{S_1,\ldots, S_k} = \sum_{T\in\mathcal S(S_1,\ldots, S_k)} \text{sign(T) } m_{T},$$
where $m_T$ is the split-trek-system monomial of the split-trek system $T=\{(P^{(1)}_1,\ldots, P^{(1)}_k), \ldots, (P^{(n)}_1,\ldots, P^{(n)}_k)\}$, defined as 
$$m_T = \prod_{i=1}^n \Phi^{(k)}_{top((P^{(i)}_1,\ldots, P^{(i)}_k)),\ldots, top((P^{(i)}_1,\ldots, P^{(i)}_k))}\prod_{j=1}^k\lambda^{P^{(i)}_j}.$$
Furthermore, the sum can be taken over the set $\widetilde{\mathcal{S}}(S_1,\ldots, S_k)$ of $k$-split-trek-systems without sided intersection.
\end{proposition}
The proof of Proposition \ref{det_T_trek_system_m_moments} can be found in Appendix \ref{Appendix D}.
We now prove an analog of Theorem~\ref{main_theorem} for the tensors of $3^{rd}$-order moments $(k=3)$.
\begin{theorem}\label{theorem_third-order_moments}
Let $S_1, S_2, S_3$ be subsets of $V$ with $\#S_1 = \#S_2 = \#S_3$. Then,
$$ \text{det} \ \mathcal{N}^{(3)}_{S_1, S_2, S_3} = 0$$ if and only if every system of $3$-split-treks between $S_1, S_2, S_3$ has a sided intersection.
\end{theorem}

\begin{proof}
In order to prove Theorem~\ref{theorem_third-order_moments}, notice that in equation (\ref{equation_T_after_Cauchy_Binet_case_moment}), $\Phi^{(k)}$ is diagonal for $k=3$. Furthermore, a 3-trek and a 3-split-trek are the same (and similarly for a 3-trek system and a 3-split-trek-system). Then the proof of Theorem~\ref{theorem_third-order_moments} follows the same reasoning as that of Theorem~\ref{main_theorem}.
\end{proof}

Notice that for $k>3$, $\Phi^{(k)}$ is not diagonal, and for that reason, we cannot easily extend  Theorem~\ref{theorem_third-order_moments} to higher-order moments. 
For higher-order moments $(k>3)$, we conjecture the following theorem by analogy with Theorem \ref{main_theorem}.
\begin{conjecture} \label{main_theorem_moments}
\noindent
Let $S_1,..., S_k$ be subsets of $V$ with $\#S_1 = ... = \#S_k$. Then,
$$ \text{det} \ \mathcal{N}^{(k)}_{S_1,..., S_k} = 0$$ if and only if every system of $k$-split-treks between $S_1,..., S_k$ has a sided intersection.
\end{conjecture}
Note that the \textit{if} direction is straightforward since we can express the determinant as a sum of split-trek monomials of $k$-split-trek systems without sided intersection, as in Proposition \ref{det_T_trek_system_m_moments}.
\noindent
Provided Conjecture~\ref{main_theorem_moments} is true, we can show the following relationship between  moment tensors of different orders.
\begin{proposition}\label{lower_order_moments}
\noindent
Consider $k\geq 4$ sets of vertices $S_1,...,S_k\subseteq V$ such that $\#S_1 = ... = \#S_k = n$. Let us suppose that the tensor of $k$-th order moments $\mathcal{N}^{(k)}_{S_1,\ldots, S_k}$ indexed by $S_1,...,S_k$ has null determinant: $\text{det }\mathcal{N}^{(k)}_{S_1,...,S_k} = 0$. Then, for any $2\leq h\leq k-2$ and any partition $\{1,\ldots, k\} = \{i_1,\ldots, i_h\}\cup \{j_1,\ldots, j_{k-h}\}$, either the determinant of the $h$-th order moment tensor $\mathcal{N}^{(h)}_{S_{i_1},\ldots, S_{i_h}}$ is zero, or the determinant of the $(k-h)$-th order moment tensor $\mathcal{N}^{(k-h)}_{S_{j_1},\ldots, S_{j_{k-h}}}$ is zero.
\end{proposition}

\begin{proof} We are going to show the contrapositive statement. Suppose that there exists $2\leq h\leq k-2$ and a partition $\{1,\ldots, k\} = \{i_1,\ldots, i_k\}\cup \{ j_1,\ldots, j_{k-h}\}$ such that both$\,$ det $\mathcal{N}^{(h)}_{S_{i_1},\ldots, S_{i_h}}\neq 0$ and$\,$ det $\mathcal{N}^{(k-h)}_{S_{j_1},\ldots, 
S_{j_{k-h}}}\neq 0$. Then, by Conjecture~\ref{main_theorem_moments} there exists an $h$-split-trek system $T_1$ with no sided intersection between $S_{i_1},\ldots, S_{i_h}$ and a $(k-h)$-split-trek system $T_2$ with no sided intersection between $S_{j_1},\ldots, S_{j_{k-h}}$. Then, combining together $T_1$ and $T_2$, we get a valid $k$-split-trek system between $S_1,\ldots, S_k$ with no sided intersection, which then implies $\text{det } \mathcal N^{(k)}_{S_1,...,S_k} \neq 0$.
\end{proof}

\begin{example}
We illustrate Proposition~\ref{lower_order_moments} for the case $k = 4$ and $h =2$.. In the graph from Figure~\ref{example_relationship_moments}(a), det $\mathcal{N}^{(4)}_{S_1,S_2,S_3,S_4} = 0$. In Figure~\ref{example_relationship_moments}(b), det $\mathcal{N}^{(2)}_{S_1,S_2} = 0$ and in Figure~\ref{example_relationship_moments}(c), det $\mathcal{N}^{(2)}_{S_3,S_4} \neq 0$.
\begin{figure}[h!]
	\centering
	\begin{subfigure}[b]{0.25\linewidth}
	\includegraphics[width=\linewidth]{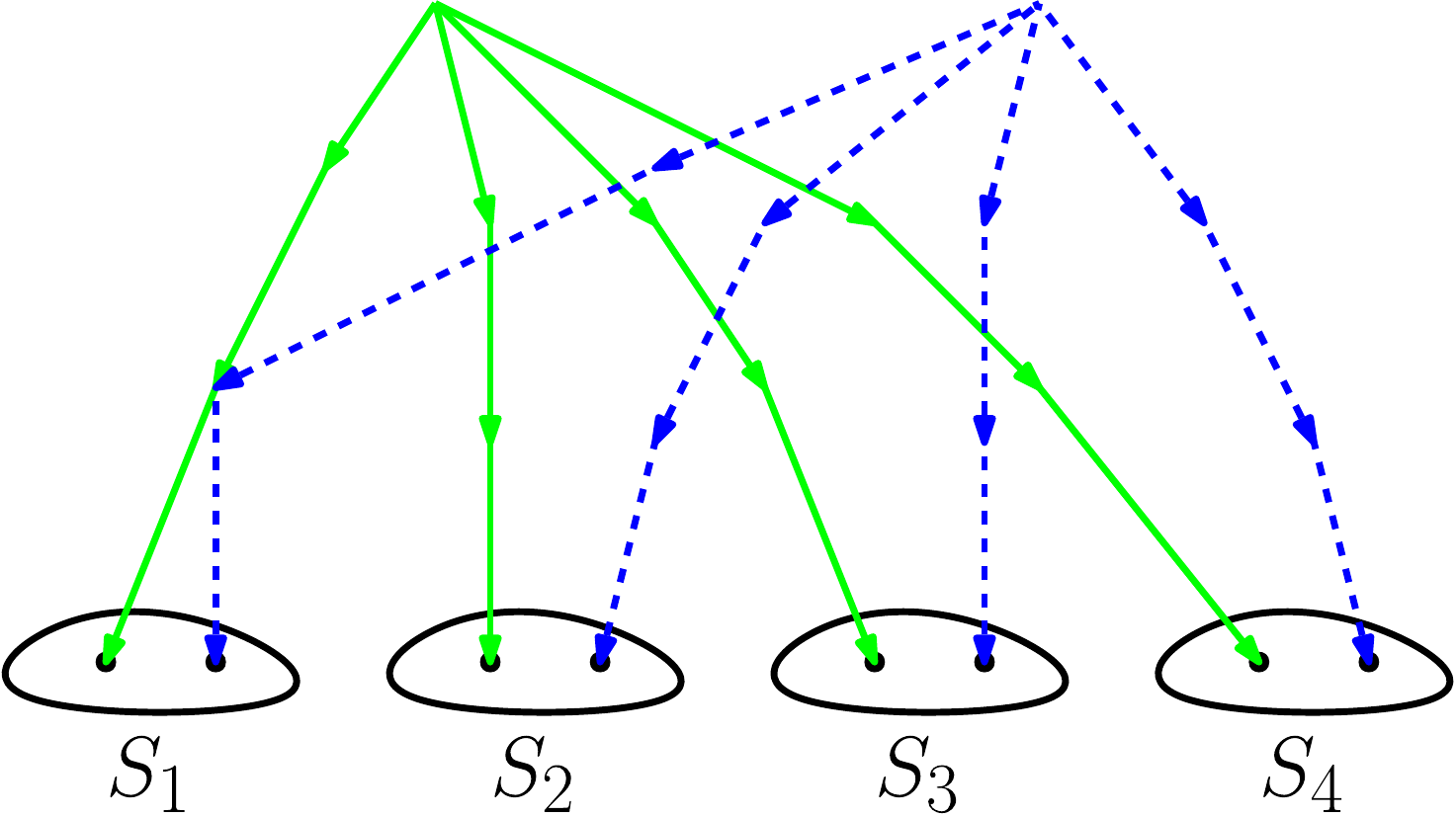}
	\caption{\label{fig1_rel_moments}}
	\end{subfigure}\hspace{0.1\textwidth}%
	\begin{subfigure}[b]{0.2\linewidth}
		\includegraphics[width=\linewidth]{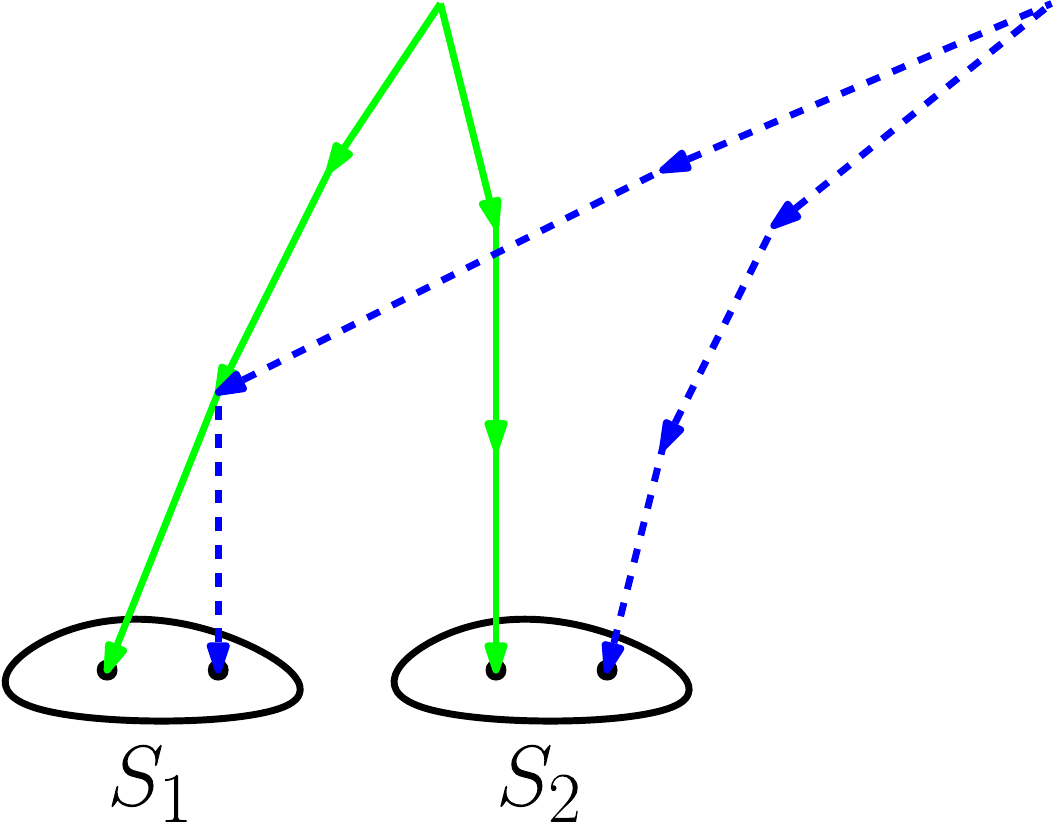}
	\caption{\label{fig2_rel_moments}}
	\end{subfigure}\hspace{0.1\textwidth}%
	\begin{subfigure}[b]{0.2\linewidth}
		\includegraphics[width=\linewidth]{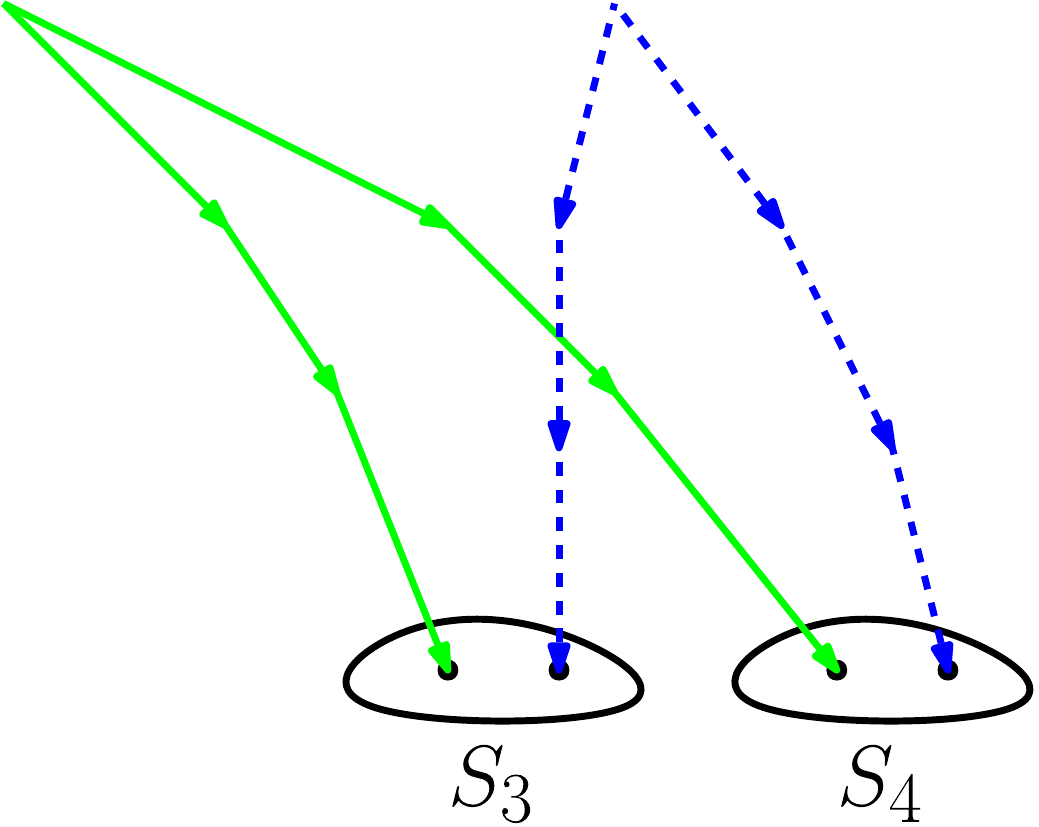}
	\caption{\label{fig3_rel_moments}}
	\end{subfigure}\hspace{0.1\textwidth}%
	\caption{}
	\label{example_relationship_moments}
\end{figure}
\end{example}

\noindent
Proposition~\ref{lower_order_moments} would give a nice relationship between the vanishing of determinants of high-order moment tensors and low-order moment tensors. One of our initial hopes was that in the case of Gaussian random variables, the vanishing of high-order moment determinants would be able to explain constraints in the model that are not subdeterminants of the covariance matrix (sometimes called Verma constraints)~\cite{Verma1991}. However, Proposition~\ref{lower_order_moments} implies that if a high-order moment determinant vanishes, then so does a lower-order one. On the other hand there are lots of Gaussian graphical models, for which there are no covariance determinants vanishing~\cite{Drton2018}, thus the vanishing of determinants would not suffice to describe the model.

\section{Conclusion}\label{section_6}

In this paper, we give implicit constraints on linear non-Gaussian structural equation models by providing a relationship between the vanishing of subdeterminants of the tensors of $k$-th order cumulants and a combinatorial criterion on the corresponding graph. Specifically, we show that the determinant of the subtensor of the $k$-th order cumulants for $k$ sets of vertices with equal cardinality vanishes if and only if there is no system of $k$-treks between these sets without sided intersection. One of the main contributions of our work is the introduction of multi-directed edges in the hidden variable case, and our multi-trek criterion which allows us to, for example, detect the presence of a common cause of multiple vertices.

A few questions for further research remain. As shown in Example \ref{example_main_theorem}, Theorem \ref{main_theorem} gives a criterion for checking if random variables have a common cause or not. It would be interesting to build a test statistic based on this criterion. Furthermore, 
our criterion can be coupled with existing algorithms to recover a mixed graph with multi-directed edges from observational data.

\noindent
\section*{Acknowledgements}
We would like to thank Mathias Drton for suggesting the problem. Elina Robeva was partially supported by an NSF Postdoctoral Fellowship (DMS 1703821).
\nocite{*}
\bibliographystyle{acm}

\appendix\label{Appendix}
\section{Proof of Cauchy-Binet for tensors}\label{Appendix A}
\noindent
We prove a tensor version of the Cauchy-Binet Theorem~\cite{Broida} (Lemma~\ref{Cauchy-Binet theorem}) for the determinant of the product $AB$, where $A$ is a tensor of $k$-th order and $B$ is a matrix. We then apply this proposition to the Tucker decomposition of tensors $
\mathcal E^{(k)}\bullet (I-\Lambda)^{-k}$ in the proof of Proposition~\ref{det_T_trek_system_m}.

\noindent
\begin{proof}[Proof of Lemma~\ref{Cauchy-Binet theorem}]
We will present the proof assuming that we multiply the matrix $B$ along the second dimension of $A$. Notice that the proof would have followed the same reasoning should we have multiplied along any other dimension. When we multiply $B$ along the second dimension of $A$, the entry $c_{i_1...i_k}$ of the product $C = AB$ is given by:
\begin{equation}
    c_{i_1...i_k} = \sum_{l} A_{i_1li_3...i_k}b_{li_2}.
\end{equation}
\noindent
Let's adopt the following notation:
\begin{itemize}
    \item $F$ is the set of functions with domain $[p]:= \{1,...,p\}$ and range $[n]:= \{1,...,n\}$
    \item $G$ is the set of functions in $F$ that are injective 
    \item $H$ is the set of strictly increasing functions that map $p$ ($p<n$) elements in $[n]$ to $p$ elements in $[n]$
    \item $\mathfrak{S}$ is the set of permutations of the elements $\{1,...,p\}$
    
We use the following equality
\begin{equation}\label{equality_CB}
    \prod_{i=1}^{p} \sum_{l} A_{il\sigma_3(i) ...\sigma_{k}(i)} B_{l \sigma_2(i)} = \sum_{f \in F} \prod_{i=1}^{p} A_{if(i)\sigma_3(i)...\sigma_{k}(i)}B_{f(i)\sigma_2(i)}
\end{equation}
    
\end{itemize}
Using Definition~\ref{definition_determinant_tensor} for the tensor determinant, we have:
\begin{equation}
\begin{aligned}
\text{det}(AB) &= \sum_{\sigma_2,...,\sigma_{k}} \text{sign} (\sigma_2) ... \text{sign} (\sigma_{k})  \prod_{i=1}^{p} (AB)_{i \sigma_2(i) ... \sigma_{k}(i)}\\
&= \sum_{\sigma_2,...,\sigma_{k}} \text{sign} (\sigma_2) ...  \text{sign} (\sigma_{k}) \prod_{i=1}^{p} \sum_{l} A_{il\sigma_3(i) ...\sigma_{k}(i)} B_{l \sigma_2(i)},\\
&\text{ from the definition of the product of a tensor by a matrix} \notag\\
&= \sum_{\sigma_2,...,\sigma_{k}} \text{sign} (\sigma_2) \cdots\text{sign} (\sigma_{k}) \sum_{f \in F} \prod_{i=1}^{p} A_{if(i)\sigma_3(i)...\sigma_{k}(i)}B_{f(i)\sigma_2(i)}, \text{from equality (\ref{equality_CB})}\\
&= \sum_{f \in F} \Bigg( \sum_{\sigma_3,...,\sigma_{k}} \text{sign}(\sigma_3)... \text{sign}(\sigma_{k}) \prod_{i=1}^{p} A_{if(i)\sigma_3(i)...\sigma_{k}(i)}\Bigg) \Bigg( \sum_{\sigma_2} \text{sign }(\sigma_2) \prod_{i=1}^{p}  B_{f(i)\sigma_2(i)}\Bigg)\\
&\text{by re-arranging terms}\notag\\
&=\sum_{f \in F} \Bigg( \sum_{\sigma_3,...,\sigma_{k}} \text{sign}(\sigma_3)... \text{sign}(\sigma_{k}) \prod_{i=1}^p A_{if(i)\sigma_3(i)...\sigma_{k}(i)}\Bigg) \text{det} (B_{f}), \\
&\text{from the definition of the determinant of a matrix, and where } B_f \text{ is the submatrix } B \text{ whose rows are selected by } f \\
&=\sum_{f \in G} \Bigg( \sum_{\sigma_3,...,\sigma_{k-1}} \text{sign}(\sigma_3)... \text{sign}(\sigma_{k}) \prod_{i=1}^{p} A_{if(i)\sigma_3(i)...\sigma_{k}(i)}\Bigg) \text{det} (B_{f}) \ \text{because} \ \text{det}(B_{f}) = 0 \ \text{for } f \notin G\\
&=\sum_{h \in H} \sum_{\gamma \in \mathfrak{S}}\Bigg( \sum_{\sigma_3,...,\sigma_{k}} \text{sign}(\sigma_3)... \text{sign}(\sigma_{k}) \prod_{i=1}^{p} A_{ih(\gamma(i))\sigma_3(i)...\sigma_{k}(i)}\Bigg) \text{det}(B_{h(\gamma)}) \\
\end{aligned}
\end{equation}
\noindent
Since $A_{ih(\gamma(i))\sigma_3(i)...\sigma_{k}(i)} = (A_{i\gamma(i)\sigma_3(i)...\sigma_{k}(i)})_h$  where $(A_{i\gamma(i)\sigma_3(i)...\sigma_{k}(i)})_h$   is the submatrix of $A$ with columns selected by $h$, and $\text{det}(B_{h(\gamma)}) = \ \text{sign} (\gamma) \ \text{det}(B_h)$, we have:
\begin{equation}
\begin{aligned}
\text{det}(AB) &= \sum_{h \in H} \Bigg(\sum_{\gamma \in \mathfrak{S}} \sum_{\sigma_3,...,\sigma_{k}} \text{sign}(\gamma) \ \text{sign}(\sigma_3)... \text{sign}(\sigma_{k}) \prod_{i=1}^{p} (A_{i\gamma(i)\sigma_3(i)...\sigma_{k}(i)})_h\Bigg) \text{det}(B_h)\\
&= \sum_{h \in H} \text{det} (A_h) \ \text{det}(B_h).
\end{aligned}
\end{equation}
\end{proof}

\section{Proof of Proposition \ref{det_T_trek_system_m_moments}}\label{Appendix D}

\begin{proof}
\noindent
By applying the tensor version of the Cauchy-Binet Theorem $k$ times to equation (\ref{equation_T_Phi_Lambda}), we get:
\begin{equation}\label{equation_T_after_Cauchy_Binet_case_moment}
\text{det }\mathcal{N}^{(k)}_{S_1,...,S_k} = \sum\limits_{R_1,...,R_k} \text{det }\Phi^{(k)}_{R_1,...,R_k} \
 \text{det} (I-\Lambda)^{-1}_{R_1,S_1} \ ... \ \text{det}(I-\Lambda)^{-1}_{R_k,S_k}
\end{equation}
Additionally, from equation (\ref{equation_T_after_Cauchy_Binet_case_moment}), we can write
\begin{equation}
\begin{aligned}
    \mathcal{N}^{(k)}_{i_1,...,i_k} &= \sum_{T \in \mathcal{S}(i_1,...,i_k)} \text{sign}(T) \ m_T,
\end{aligned}
\end{equation}
where $m_T$ is the $k$-split-trek monomial defined by $m_T = {\phi_{top(P_{1},..., P_{k})} \lambda^{P_{1}}} ....  \lambda^{P_{k}}$.

Assuming that $\#R_1= ...=\#R_k = \#S_1 = ... = \#S_k = n$, we then get:
\begin{equation*}
\begin{aligned}
\text{det } \mathcal{N}^{(k)}_{S_1,...,S_k} \ &= \sum_{\sigma_2 \in \mathfrak{S}_{n_{2}},..., \sigma_{k} \in \mathfrak{S}_{n_{k}}} \Bigg(\sum_{\substack{T_1 \in \mathcal{S}(s_1, s_{\sigma_2(1)},
...,s_{\sigma_{k}(1)})\\
...\\
...\\
T_n \in \mathcal{S}(s_n, s_{\sigma_2(n)},...,s_{\sigma_{k}(n)})}} \text{sign}(\sigma_2) \ldots \text{sign}(\sigma_{k})\ m_{T_1}...m_{T_n} \Bigg)\\
&= \sum_{\substack{\sigma_i \in \mathfrak{S}_{n_i}\\
i\in \{2,\ldots,k\}}} \Bigg(\sum_{\substack{T_j \in \mathcal{S}(s_j, s_{\sigma_2(j)},
...,s_{\sigma_{k}(j)})\\
j \in \{1,\ldots,n\}}}
 \prod_{l = 2}^{k}\text{sign}(\sigma_l)\prod_{s = 1}^{n} m_{T_s}\Bigg)\\
&= \sum_{T\in\mathcal S(S_1,\ldots, S_k)} \text{sign}(T) \  m_{T}
\end{aligned}
\end{equation*}

\noindent
where $\mathfrak{S}_{n_{i}}$ is the set of permutations of the nodes in $S_{i+1}$, $T$ runs over all $k$-split-trek systems between $S_1,...,S_k$ and sign($T$) = $\text{sign}(\sigma_2) \ ...\ \text{sign}(\sigma_{k})$. In this expression, we have $m_{T} = \prod_{x=1}^{n} m_{T_x}$, where $m_{T_x}=\Phi_{top(P_{s_{1}},...,P_{s_{k}})}\lambda^{P_{s_{1}}}...\lambda^{P_{s_{k}}}$, i.e., $m_{T}$ is the product of the monomials of the $n$ $k$-split-treks that form the $k$-split-trek-system $T$.

Similarly to the proof of Proposition~\ref{det_T_trek_system_m}, we can use the Gessel-Viennot-Lindstrom Lemma to show that the sum in the expression of det $\mathcal N_{S_1,\ldots, S_k}^{(k)}$ can be taken over the set $\tilde{\mathcal S}(S_1,\ldots, S_k)$ of $k$-split-trek systems between $S_1,\ldots, S_k$ without sided intersection.
\end{proof}
\end{document}